\documentclass[11pt]{article}
\usepackage[utf8]{inputenc}
\usepackage{verbatim,amsmath,eufrak}
\usepackage{latexsym,amssymb,amsxtra,amsfonts,xcolor, bbm}
\usepackage{cite} 
\usepackage{makeidx}
\usepackage[hypertexnames=false,hyperfootnotes=false,colorlinks=true,linkcolor=blue,%
citecolor=purple,filecolor=magenta,urlcolor=cyan,unicode,linktocpage=true,pagebackref=false]{hyperref}

\renewcommand{\arraystretch}{1.2}
\makeatletter
\newdimen\normalarrayskip              
\newdimen\minarrayskip                 
\normalarrayskip\baselineskip
\minarrayskip\jot
\newif\ifold             \oldtrue            \def\new{\oldfalse}
\def\arraymode{\ifold\relax\else\displaystyle\fi} 
\def\eqnumphantom{\phantom{(\theequation)}}     
\def\@arrayskip{\ifold\baselineskip\z@\lineskip\z@
     \else
     \baselineskip\minarrayskip\lineskip2\minarrayskip\fi}
\def\@arrayclassz{\ifcase \@lastchclass \@acolampacol \or
\@ampacol \or \or \or \@addamp \or
   \@acolampacol \or \@firstampfalse \@acol \fi
\edef\@preamble{\@preamble
  \ifcase \@chnum
     \hfil$\relax\arraymode\@sharp$\hfil
     \or $\relax\arraymode\@sharp$\hfil
     \or \hfil$\relax\arraymode\@sharp$\fi}}
\def\@array[#1]#2{\setbox\@arstrutbox=\hbox{\vrule
     height\arraystretch \ht\strutbox
     depth\arraystretch \dp\strutbox
     width\z@}\@mkpream{#2}\edef\@preamble{\halign
\noexpand\@halignto
\bgroup \tabskip\z@ \@arstrut \@preamble \tabskip\z@ \cr}%
\let\@startpbox\@@startpbox \let\@endpbox\@@endpbox
  \if #1t\vtop \else \if#1b\vbox \else \vcenter \fi\fi
  \bgroup \let\par\relax
  \let\@sharp##\let\protect\relax
  \@arrayskip\@preamble}
\def\eqnarray{\stepcounter{equation}%
              \let\@currentlabel=\theequation
              \global\@eqnswtrue
              \global\@eqcnt\z@
              \tabskip\@centering
              \let\\=\@eqncr
 \halign to \displaywidth\bgroup
    \eqnumphantom\@eqnsel\hskip\@centering
    $\displaystyle \tabskip\z@ {##}$%
    \global\@eqcnt\@ne \hskip 2\arraycolsep
         $\displaystyle\arraymode{##}$\hfil
    \global\@eqcnt\tw@ \hskip 2\arraycolsep
         $\displaystyle\tabskip\z@{##}$\hfil
         \tabskip\@centering
    &{##}\tabskip\z@\cr}
\begingroup\ifx\undefined\newsymbol \else\def\input#1 {\endgroup}\fi

\def\be{\begin{eqnarray}}
\def\ee{\end{eqnarray}}
\def\nn{\nonumber}

\def\p{\partial}

\def\beq{\begin{equation}}
\def\eeq{\end{equation}}
\def\ba{\beq\new\begin{array}{c}}
\def\ea{\end{array}\eeq}
\def\be{\ba}
\def\ee{\ea}

\def\Tr{{\rm Tr}\,}

\def\diag{{\rm diag}\,}

\newfont{\Bbbb}{msbm7 scaled 1\@ptsize00}
\newcommand{\ZZ}{\mathbb{Z}}
\newcommand{\z}{\raise-1pt\hbox{$\mbox{\Bbbb Z}$}}
\newcommand{\cc}{\raise-1pt\hbox{$\mbox{\Bbbb C}$}}
\newcommand{\rr}{\raise-1pt\hbox{$\mbox{\Bbbb R}$}}

\def\lbr{\left <}
\def\rbr{\right >}

\def\Gr{{\rm Gr}\,}
\def\spn{{\rm span}\,}

\def\deg{\operatorname{deg} }
\newcommand{\ben}{\begin{eqnarray*}}
\newcommand{\een}{\end{eqnarray*}}

\def\QQ{\mathbb{Q}}
\def\ZZ{\mathbb{Z}}
\def\RR{\mathbb{R}}
\def\CC{\mathbb{C}}
\def\PP{\mathbb{P}}

\def\ii{\mathbf{i}}
\def\tt{\mathbf{t}}
\def\qq{\mathbf{q}}

\def\H{\mathcal{H}}
\def\F{\mathcal{F}}

\newfont{\alef}{msbm10 at 11pt}
\newfont {\goth}{eufm10 at 11pt}
\def\mathbb#1{\hbox{{\alef #1}}}

\let\@@savethanks\thanks
\def\thanks#1{\gdef\thefootnote{\alph{footnote}}\@@savethanks{#1}}

\newtheorem{theorem}{Theorem}[section]
\newtheorem{lemma}[theorem]{Lemma}
\newtheorem{proposition}[theorem]{Proposition}
\newtheorem{corollary}[theorem]{Corollary}
\newtheorem{remark}[theorem]{Remark}

\newenvironment{proof}[1][Proof.]{\begin{trivlist}
\item[\hskip \labelsep {\bfseries #1}]}{\end{trivlist}}

\makeatletter
\g@addto@macro \normalsize {%
 \setlength\abovedisplayskip{14pt plus 3pt minus 3pt}%
 \setlength\belowdisplayskip{14pt plus 3pt minus 3pt}%
  \setlength\abovedisplayshortskip{11pt plus 3pt minus 3pt}%
 \setlength\belowdisplayshortskip{11pt plus 3pt minus 3pt}%
}
\makeatother

\newcommand{\Addresses}{{
\bigskip
\footnotesize
\textsc{Center for Geometry and Physics, Institute for Basic Science (IBS), Pohang 37673, Korea}\par\nopagebreak

E-mail: \texttt{alexandrovsash at gmail.com}

\medskip

\textsc{Kavli IPMU (WPI), UTIAS, 
The University of Tokyo, Kashiwa,
Chiba 277-8583, Japan }\par\nopagebreak

E-mail: \texttt{todor.milanov@ipmu.jp}
}}

\title{ Matrix model for the total descendant potential of a simple
    singularity of type $D$ }

\author{
Alexander Alexandrov
\and 
Todor Milanov
}

\date{
\emph{To the memory of Boris Dubrovin, whose work will continue to
  inspire }}

\begin{document}

\maketitle

\begin{abstract}
\centerline{{\sc Abstract}}
\bigskip
We construct a Hermitian matrix model for the total descendant potential of a
simple singularity of type D similar to the Kontsevich matrix model
for the generating function of intersection numbers on the
Deligne--Mumford moduli spaces $\overline{\mathcal{M}}_{g,n}$.
\end{abstract}

\setcounter{footnote}{0}
\setcounter{tocdepth}{3}
\def\thefootnote{\arabic{footnote}}
\tableofcontents

\section{Introduction}

The motivation for our work comes from the celebrated
conjecture of Witten \cite{W} saying that the intersection
theory on the Deligne--Mumford moduli spaces $\overline{\mathcal{M}}_{g,n}$
is governed by the KdV hierarchy. The conjecture was first proved by 
Kontsevich \cite{K}. One of the key
ingredients in Kontsevich's argument is a certain Hermitian matrix model, which is
now known as the {\em Kontsevich matrix model}. The latter provides
both a tau-function solution to the KdV hierarchy and a Feynman diagram expansion
which can be related to the intersection theory on the moduli spaces
$\overline{\mathcal{M}}_{g,n}$. 
Witten's conjecture can be generalized in various ways, such
as Gromov--Witten theory \cite{W} and
Fan--Jarvis--Ruan--Witten (FJRW)
theory \cite{FJRW}.   It is natural to ask to what extent the
Kontsevich matrix model can be generalized too. There are only two
cases in which the answer to the above question is known to be positive, that is,
Gromov--Witten theory of the projective line $\PP^1$ and FJRW theory of the simple singularities of
type $A$. The case of a simple singularity of type $D$ is arguably the next on
the list and this is exactly the construction carried out in this
paper. 

We expect that there exists a matrix model for the FJRW invariants
corresponding to the remaining simple singularities $E_N$
($N=6,7,8$). However, at this point the construction of 
such a model looks quite challenging. The main issue from our point of
view is that the fermionic realizations of the  basic representations of the affine Lie algebras of
type $E_N$ ($N=6,7,8$) are not known. 
In the rest of this introduction let us concentrate on stating our
results.

\subsection{$(h,2)$-reduction of the 2-component BKP hierarchy}
The explicit form of the Kac--Wakimoto hierarchies of type D was
determined by ten Kroode and van de Leur \cite{KvL}. We are
interested in the principal case, that is, the Kac--Wakimoto hierarchy
corresponding to the conjugacy class of the Coxeter transformation. It
turns out that the principal Kac--Wakimoto hierarchy of type D is a
reduction of the so-called 2-component BKP hierarchy (see \cite{LWZ}
and Corollary 1 in \cite{CM}). 

Let
\be\nn
\mathbf{t}_a:=(t_{a,m})_{m\in \z_{\rm odd}^+}, \quad a=1,2
\ee
be two sequences of formal variables. We denote by $\ZZ_{\rm odd}$
the set of all odd integers and by $\ZZ_{\rm odd}^+$ the set of all
positive odd integers. Suppose that $h:=2N-2>0$ is an even integer. A formal power series 
\be\nn
\tau(\mathbf{t}_1,\mathbf{t}_2)\in \CC[\![\mathbf{t}_1,\mathbf{t}_2]\!] 
\ee
is said to be a {\em tau-function} of the $(h,2)$-reduction of the
2-component BKP hierarchy if the following Hirota bilinear equations
hold: 
\be\nn
\Omega_m(\tau\otimes \tau)=0,\quad m\geq 0.
\ee
Here $\Omega_m$ is the following bi-linear operator acting on
$\CC[\![\mathbf{t}_1,\mathbf{t}_2]\!] ^{\otimes 2}$
\be\nn
\operatorname{Res}_{z=0} \frac{dz}{z}
\Big(
z^{hm} \Gamma_1(\mathbf{t},z) \otimes \Gamma_1(\mathbf{t},-z) -
z^{2m} \Gamma_2(\mathbf{t},z) \otimes \Gamma_2(\mathbf{t},-z)
\Big),
\ee
where 
\be\nn
\Gamma_a(\mathbf{t}, z) := \exp\Big(
\sum_{m\in \z_{\rm odd}^+} t_{a,m} z^m\Big)
\exp\Big(-
\sum_{m\in \z_{\rm odd}^+} 2\partial_{t_{a,m}} \frac{z^{-m}}{m}\Big)
\ee
are {\em vertex operators}. 

Our main interest is in tau-functions satisfying the so-called
{\em string equation} 
\be\label{str_eqn}
L_{-1}\, \tau=0,\quad \quad 
L_{-1} := -\ii\, \frac{\partial}{\partial t_{1,1}} +
\sum_{a=1,2}\sum_{m\in \z_{\rm odd}}
: J^a_m J^a_{-m- h_a}:,
\ee
where $\ii:=\sqrt{-1}$, $h_1:=h:=2N-2$, $h_2=2$, 
\be\label{Jam}
J^a_m := 2\frac{\partial}{\partial t_{a,m}},\quad
J^a_{-m} := m t_{a,m},\quad m\in \ZZ_{\rm odd}^+,
\ee
and the normal ordering means that the annihilation operators, i.e.,
all $J^a_m$ with $m>0$, should be applied first. 
The existence of a tau-function satisfying the string equation was
proved by Vakulenko \cite{V}. 
Cheng and Milanov proved \cite{CM} that the total descendant
potential of $D_N$-singularity is also a solution of the
$(h,2)$-reduction satisfying the string equation. Vakulenko did not
discuss the uniqueness of his construction, while Cheng and Milanov
proved that the total descendant potential satisfies also a dilaton
constraint and that the string equation and the dilaton constraint
uniquely determine the tau-function. Our first result is 
that the dilaton constraint is redundant, that is, the following
theorem holds:
\begin{theorem}\label{t1}
  There exists a unique tau-function of the $(h,2)$-reduction of the 2-component BKP
  hierarchy satisfying the string equation (\ref{str_eqn}).
\end{theorem}
This result is a direct analog of the well-known statement for the
$A_N$ singularities (for example see \cite{KS}). The problem of
characterizing tau-functions of the Drinfeld--Sokolov hierarchies via
the string equation was studied in a recent paper by
Cafasso and Wu  \cite{CaWu}. We expect that Theorem \ref{t1} can also be derived from their results.
Theorem \ref{t1} is
very important for identifying tau-functions of the 2-component BKP
hierarchy that have different origin. In fact, the unique tau-function
specified by Theorem \ref{t1} arises in at least 3 other different
ways: the total descendent potential of $D_N$-singularity,
a tau-function of the principal Kac--Wakimoto hierarchy of type $D_N$,
and the generating function of FJRW-invariants of the Berglund--H\"ubsch
dual singularity $D_N^T$. We refer to Section \ref{sec:geometry} for
more details. Let us point out that our matrix model will be obtained by using
the theory of the 2-component BKP hierarchy. Section
\ref{sec:geometry} is written only for people who might be interested
in the applications of our result to geometry. Especially, the FJRW theory might provide a geometric approach
to our model similar to Kontsevich's argument in \cite{K}. In other words,
our matrix model could be viewed as a motivation to look for a
combinatorial model for the virtual fundamental cycle in the FJRW
theory of the Berglund--H\"ubsch dual singularity $D_N^T$. 

In the rest of this paper, except for Section \ref{S_3}, we will be working with the unique
tau-function $\tau(\mathbf{t}_1,\mathbf{t}_2)$ specified by Theorem
\ref{t1}. Sometimes, we denote it also by
$\tau^{\rm CM}(\mathbf{t}_1,\mathbf{t}_2)$ in order to emphasize that we follow
the normalization of Cheng--Milanov \cite{CM}.  
The precise identification between 
$\tau^{\rm  CM}(\mathbf{t}_1,\mathbf{t}_2)$, 
the total descendant potential of $D_N$-singularity, and the
generating function of FJRW-invariants of the Berglund--H\"ubsch dual
$D_N^T$ is given in Section \ref{sec:BKP-D}.

\subsection{Matrix model}\label{sec:mm}
Let us introduce two non-degenerate diagonal matrices
\be\nn
Z_a=\diag(z_{a,1},\dots,z_{a,N_a}),\quad a=1,2,
\ee
such that, 
$\operatorname{Arg}(z_{1,i}) =\tfrac{\pi}{2(h+1)}$,
$\operatorname{Arg}(z_{2,j}) =0$, that is, $z_{2,j}\in \RR_{> 0}$, and
$N_1+N_2$ is an even number. We will refer to the substitution
\be\nn
t_{a,m} = -\frac{2}{m} \, \Tr(Z_a^{-m}) = -2 \sum_{i=1}^{N_a}
\frac{(z_{a,i})^{-m}}{m}
\ee
as the  {\em Miwa parametrization}, while the formal series 
$\tau(Z_1,Z_2):=\tau(\mathbf{t}_1,\mathbf{t}_2)|_{t_{a,m}:=
-\tfrac{2}{m}\operatorname{Tr}(Z_a^{-m})}$ 
will be called the tau-function in the Miwa parametrization. 

Let us denote by ${\mathcal H}_N$ the linear space of all $N\times N$  Hermitian matrices, and by  
${\mathcal H}_N^+$ the space of all positive definite $N\times N$  Hermitian matrices.
The space $\mathcal{H}_N$ is equipped with a canonical measure 
\begin{align}
\nn
  \left[d X\right] & := p_*\left(
                     \Delta_{N}(x)^2  \left[d U\right] \,
                     \prod_{i=1}^{N} d x_i
                     \right),
\end{align}
where for a diagonal matrix $a=\operatorname{diag}(a_1,\dots,a_N)$ we
denote by 
$\Delta_N(a):=\prod_{1\leq i<j\leq N} (a_j-a_i)$ the Vandermonde
determinant, $\left[d U\right] $ is the Haar measure on the unitary
group $U(N)$, and $p_*$ is the pushforward operation, i.e.,
integration along the fiber, with respect to the
proper surjective map $p:U(N)\times \RR^N \to \H_N$,
$p(U,x):= U \diag(x_1,x_2,\dots,x_N)U^\dagger$. 
For our purposes, it is more convenient to work with the following measure:
\begin{align}
\nn
\widetilde{\left[d X\right]} & :=\frac{\left[d X\right]}{\sqrt{\det\left(X\otimes I_N+I_N\otimes X\right)}},
\end{align}
where $I_N$ is the identity matrix of size $N$.

Let us introduce also the interaction term
\be\nn
S(X,Y)=\det\left(\frac{X\otimes I_{N_2}-I_{N_1}\otimes Y}{X\otimes I_{N_2}+I_{N_1}\otimes Y}\right),
\ee
and the potential
\be\label{eq_Wpot}
W(X,Z)= \frac{{\bf i}^h}{h+1}X^{h+1}-XZ.
\ee
The main result of this paper can be stated as follows. 
\begin{theorem}\label{t2}
  Under the above notation, the tau-function in the Miwa parametrization
  coincides with the asymptotic expansion of a matrix integral:  
\be
\nn
\tau(Z_1,Z_2)\sim \frac{e^{\frac{\ii h}{h+1} \Tr Z_1^{h+1}}}{{\cal N}} 
\int _{e^{\frac{(h+2)\pi}{2(h+1)}\ii}{\mathcal H}_{N_1}}  \widetilde{\left[d X\right]} \, 
e^{\Tr  W(X,Z_1^h) } 
\int _{{\mathcal H}_{N_2}^+} \widetilde{\left[d Y\right]}\,  
e^{\Tr W(Y,Z_2^2)}  S(X,Y),
\ee
where the normalization factor $\mathcal{N}$ is given explicitly by \eqref{formula-N}.
\end{theorem}
The asymptotic equality in Theorem \ref{t2} is interpreted via the
formal expansion of the matrix integrals according to the steepest
descent method, see Section \ref{sec:di} for more details. Establishing
the analytic properties of our matrix integral seems to be a separate
project, so we do not pursue it in this paper.  Let us just make a
comment about a possible relation to the notion of {\em strongly
  asymptotically developable} functions (see \cite{Ma} for some
background). We expect that  
the matrix integral in Theorem \ref{t2}, after choosing appropriate
integration cycles, is analytic and strongly asymptotically
developable as $z=(z_{a,j})\in (\PP^1)^{N_1+N_2}$ tend to the divisor
$\prod_{a,j} z_{a,j}=\infty$ in an appropriate multi-sector
and that $\tau(Z_1,Z_2)$ is the formal series of the corresponding
asymptotic expansion.  Note however, that since the interacting term
$S(X,Y)$
is a meromorphic function, the multivariable asymptotics that we need
are slightly more complicated than the ones in \cite{Ma}. 

\begin{remark}
The symmetry algebra of the $BKP$ hierarchy is a certain version of the central extension of the Lie algebra $B_\infty$. Because of that it might be a bit surprising that topological solution of the
2-component $BKP$ hierarchy is described in terms of a Hermitian
(not orthogonal) matrix integral. 
\end{remark}

\begin{remark}
While the integrals over $X$ and $Y$ looks pretty similar, they are
essentially different. The integral over $X$ is an asymptotic
expansion at the vicinity of the critical point $X=\ii Z_1$. It is
typical for the generalized Kontsevich models, in particular, to the 
integral for the $A_N$ singularity, that can be described as a
perturbative expansion of the Gaussian integral.  
The integral over $Y$ is an asymptotic expansion at the vicinity of
the point $Y=0$. It is essentially the matrix Laplace transform and,
in a certain sense, is much simpler.  
\end{remark}

\subsection{Organization of the paper}
In Section \ref{sec:geometry} we give a precise identification of our
tau-function with the total descendent potential of the simple
singularity of type $D_N$ and the generating function of
FJRW-invariants of the Berglund--H\"ubsch dual singularity
$D_N^T$. The reader not interested in the applications to geometry
could skip this Section, because it is not logically connected with
the rest of the paper. 

In Section \ref{S_3} we introduce the main tools of this paper, i.e.,
the formalism of neutral fermions and prove that in the Miwa
parametrization every tau-function of the 
2-component BKP hierarchy can be written as a ratio of two Pfaffians, where the numerator is given by a Pfaffian of a matrix
whose entries are given by 2-point correlators. This formula is a direct analog of the well-known determinant description of the KP tau-function in the Miwa parametrization. 
In Section
\ref{sec:Gr} we give a proof of Theorem \ref{t1}. Our strategy is
similar to \cite{KS}, that is, first, we recall the
Grassmannian description of the 2-component BKP hierarchy. The string
equation yields a certain ordinary differential equation that can be
solved in terms of certain steepest descent
integrals. Then we prove that the subspace corresponding to the point
of the BKP Grassmannian parametrizing the tau-function of interest has
a basis which can be reconstructed uniquely from the steepest descent
asymptotic of the one-dimensional integrals. Finally, in
Section \ref{sec:mi} we extend the methods from Section \ref{sec:Gr}
in order to obtain formulas for the entries of the Pfaffian matrix
from Section \ref{S_3} in terms of asymptotic expansion of certain
double integrals. This is done in Proposition \ref{kern_basis}, which
could be viewed as the key step in proving Theorem \ref{t2}. Let us
also point out that in order to prove Proposition \ref{kern_basis} we
had to use a certain symmetry of the tau-function (see Lemma \ref{le:symm}),
which on the other hand is established via Theorem \ref{t1}. In other
words, Theorem \ref{t1} plays an important role in the proof of
Theorem \ref{t2}! Finally, in the last section we make some further
remarks and outline some further direction for investigation. 

{\bf Acknowledgements.}
The work of A.A. is partially supported by IBS-R003-D1 and by RFBR grant 18-01-00926.
The work of T.M. is partially supported by JSPS Grant-In-Aid (Kiban C)
17K05193 and by the World Premier International Research Center
Initiative (WPI Initiative),  MEXT, Japan. T.M. would like to thank
Emil Horozov and Mikhail Kapranov for useful conversations on matrix
models and multivariable asymptotics. A.A. and T.M. would like to
thank anonymous referees for the suggested improvements.

\section{Geometric interpretation of the tau-function}\label{sec:geometry}

The goal of this section is to give the precise identification between
the tau-function from Theorem \ref{t1}, the total descendant potential
of a simple singularity of type $D$, and the generating function of
the FJRW invariants of type $D^T$, where $D^T$ denotes the so-called
Berglund--H\"ubsch (BH) dual. 

\subsection{FJRW-invariants}\label{sec:FJRW-inv}
The BH dual singularity $D_N^T$ is given by the polynomial $W(x,y):=x^{N-1} y
+y^2$. Note that this polynomial is not a singularity of type $D_N$.
According to \cite{FJRW}, the intersection theory on the moduli space
of $W$-spin curves corresponding to the potential $W(x,y)= x^{N-1} y
+y^2$ yields a Cohomological Field Theory (CohFT) on an
$N$-dimensional vector space 
\be\nn
\H_W:=\mathbb{C} e_0 \bigoplus \bigoplus_{i=1}^{N-1} \mathbb{C} e_{2i-1},
\ee
that has the following properties: Let us assign degrees
\be\nn
\deg(e_0)=\frac{N-2}{2N-2},\quad \deg(e_{2i-1}) = \frac{i-1}{N-1}
\quad (1\leq i\leq N-1).
\ee
\begin{enumerate}
\item[(i)] {\em Dimension constraint}: if $\alpha_i\in \H_W$ are homogeneous
  elements, then the correlator
  \begin{equation}\label{correlator}
  \lbr \alpha_1\psi^{k_1},\dots,\alpha_m\psi^{k_m}\rbr^{\rm FJRW}_{g,m}
  \end{equation}
  is non-zero only if
  \be\nn
  \sum_{i=1}^m (\deg(\alpha_i) + k_i) = 3g-3 + m + D(1-g),
  \ee
  where $D:=1-\tfrac{1}{N-1}$ is the {\em conformal dimension} of the
  CohFT. Let us point out that $D=1-\tfrac{2}{h}$, where $h=2N-2$ is
  the Coxeter number for the root system of type $D_N$.
\item[(ii)] {\em Euler characteristics constraint}:
  Put $\Theta^0:=(0,0)$, $\Theta^{2i-1} = (\tfrac{2i-1}{h},\tfrac{1}{2})$
  ($1\leq i\leq N-1$). Then the correlator (\ref{correlator}) is non
  zero only if
  \be\nn
    \frac{1}{h} (2g-2+m) -
    (\Theta^{\alpha_1}_1 +\cdots + \Theta^{\alpha_m}_1)  \in \mathbb{Z}, \\
    \frac{1}{2} (2g-2+m) -
    (\Theta^{\alpha_1}_2 +\cdots + \Theta^{\alpha_m}_2) 
      \in \mathbb{Z},
  \ee
  where $\Theta^\alpha_s$ for a homogeneous element $\alpha\in \mathbb{C}
  e_i$ stand for the $s$-component of $\Theta^i$.
\item[(iii)] Genus-0 three point correlators (see \cite{FJRW}, Section
  5.2.2):

  a) 
  All 3-point correlators $\lbr \alpha_1,\alpha_2,\alpha_3\rbr^{\rm FJRW}_{0,3}$
  satisfying the dimension constraint
  $\deg(\alpha_1)+\deg(\alpha_2)+\deg(\alpha_3)=D$ are equal to 1
  unless one of the insertions $\alpha_i\in \mathbb{C} e_0$. 

  b) The only non-zero 3-point correlator $\lbr
  \alpha_1,\alpha_2,\alpha_3\rbr^{\rm FJRW}_{0,3}$, involving an insertion
  $\alpha_i\in \mathbb{C}e_0$ is
  \be\nn
  \lbr e_0, e_0,e_1\rbr^{\rm FJRW}_{0,3} = -\frac{1}{N-1}. 
  \ee

\item[(iv)] The 4-point correlator (see \cite{FJRW}, Section 6.3.7):
\be\nn
\lbr e_3, e_3, e_{2N-5},e_{2N-3}\rbr^{\rm FJRW}_{0,4} = \frac{1}{h}.
\ee

\item[(v)] $e_1$ is a unit of the CohFT.
\end{enumerate}
Fan--Jarvis--Ruan prove the following reconstruction result (see
\cite{FJRW}, Theorem 6.2.10, part (3)):
properties (iii) and (iv) uniquely determine the CohFT, that is, the
correlators (\ref{correlator}) are uniquely determined from (iii),
(iv), and the axioms of a CohFT.

Let us recall the total descendant potential of the FJRW theory, that
is, the following generating function of FJRW-invariants:
\be\nn
\mathcal{D}^{\rm FJRW}(\hbar,\mathbf{t}^{\rm FJRW}):=
\exp\left(
\sum_{g,\kappa} 
\hbar^{g-1} \lbr 
e_{i_1} \psi^{k_1},\dots , e_{i_m} \psi^{k_m}
\rbr_{g,m} 
\frac{
t^{\rm FJRW}_{k_1,i_1}\cdots t^{\rm FJRW}_{k_m,i_m}
}{m!}
\right),
\ee
where the sum is over all integers $g\geq 0$ and over all sequences
$\kappa=((k_1,i_1),\dots,(k_m,i_m))$ of pairs.

\subsection{Mirror symmetry}\label{Mirror}
Let us recall the mirror symmetry result of Fan--Jarvis--Ruan (see
\cite{FJRW}, Theorem 6.1.3, part (3)). Let 
\be\nn
f(x_1,x_2,x_3):= x_1^2 x_2-x_2^{N-1} + x_3^2
\ee
be the $D_N$ singularity. Note that $f$ is quasi-homogeneous: if
we assign degrees $\tfrac{N-2}{2N-2}, \tfrac{1}{N-1},$ and
$\tfrac{1}{2}$ to respectively $x_1, x_2,$ and $x_3$, then $f$ is
homogeneous of degree 1. Therefore, the {\em local algebra}
\be\nn
H_f:=\CC[x_1,x_2,x_3]/(\p_{x_1}f,\p_{x_2}f,\p_{x_3}f)
\ee
is naturally a graded vector space. Let us recall also the so-called
{\em Grothendieck residue} on $H_f$.   
Note that the determinant of the Hessian matrix
$\operatorname{Hess}(f):=\operatorname{det}(\p^2f/\p x_i\p x_j)$ is a 
homogeneous polynomial of degree $D=\tfrac{N-2}{N-1}$. The maximal
possible degree of a homogeneous subspace in the local algebra is $D$.
The corresponding homogeneous subspace is 1 dimensional and hence it
is spanned by the class $[\operatorname{Hess}(f) ]$. Here
if $\psi\in \CC[x_1,x_2,x_3]$, then $[\psi]$ denotes the equivalence
class of $\psi$ in the local algebra. Given an element $\psi$ in the
local algebra, let us define the Grothendieck residue
$\operatorname{Res}(\psi)$ by the formula 
$N \psi_D= \operatorname{Res}(\psi)  [\operatorname{Hess}(f)]$, where $\psi_D$ is the
homogeneous component of $\psi$ of maximal degree. Alternatively, the
Grothendieck residue coincides with the multidimensional residue
\be\nn
\operatorname{Res}(\psi) = \operatorname{Res}\
\frac{\psi dx_1 \wedge dx_2\wedge dx_3}{\p_{x_1}(f)\p_{x_2}(f)
  \p_{x_3}(f)}, 
\ee
that is, the LHS above is by definition $(2\pi \mathbf{i})^{-3}\times$
the integral of the meromorphic 3-form along an appropriate toroidal
cycle $|\p_{x_1}(f)|=|\p_{x_2}(f)|= |\p_{x_2}(f)|=\epsilon$.
A direct computation yields $[\operatorname{Hess}(f) ]=[-4N
x_1^2]$. Let us fix a basis of $H_f$ represented by the monomials
$\phi_i(x)=x_2^{i-1}$ ($1\leq i\leq N-1$) and $\phi_N(x)=2x_1$. Then
the residue pairing takes the form
\be\nn
\operatorname{Res}(\phi_i \phi_j) = -\frac{1}{2h}\delta_{i+j,N},\quad
1\leq i,j\leq N-1, \\
\operatorname{Res}(\phi_i\phi_N) = -\delta_{i,N},\quad
1\leq i\leq N.
\ee
Using Saito's theory of primitive forms, we can define a Frobenius
structure on the space of miniversal deformations of $f$ with
primitive form $\omega=dx_1\wedge dx_2\wedge dx_3$.
Givental's higher genus reconstruction yields a total descendant
potential  
\be\nn
\mathcal{D}^{\rm SG}(\hbar,\tt^{\rm SG}) =\exp \left(
\sum_{g,\kappa} \hbar^{g-1} \lbr \phi_{i_1}\psi^{k_1},\dots,
\phi_{i_m}\psi^{k_m}\rbr_{g,m}^{SG} 
\frac{
t^{\rm SG}_{k_1,i_1}\cdots t^{\rm SG}_{k_m,i_m}
}{m!}\right).
\ee
Put $c:=\tfrac{\ii}{\sqrt{2h}} \, 2^{-(N-2)/h}$  and let us define the
following map:  
\be\nn
\operatorname{Mir}: H_f \to \H_W \\
\phi_i:= x_2^{i-1}\mapsto 2^{\tfrac{i-1}{N-1}} \, e_{2i-1}\quad (1\leq
i\leq N-1) \\
\phi_N:= 2x_1\mapsto -h\mathbf{i}\, 2^{\tfrac{N-2}{2N-2}} \, e_{0},
\ee
where $\mathbf{i}:=\sqrt{-1}$. Mirror symmetry between the
FJRW-invariants and the SG-invariants can be stated as follows:
\be\label{FJRW=SG}
c^{2g-2}
\lbr \phi_{i_1}\psi^{k_1},\dots,\phi_{i_m}\psi^{k_m}
\rbr^{\rm SG}_{g,m} = \lbr
\operatorname{Mir}(\phi_{i_1})\psi^{k_1},\dots,
\operatorname{Mir}(\phi_{i_m})\psi^{k_m}
\rbr^{\rm FJRW}_{g,m}.
\ee
Let us sketch the proof of the above formula. The idea is to establish
first two special cases. The remaining identities follow from the
special cases and the
axioms of CohFT according to the reconstruction
theorem of Fan--Jarvis--Ruan (see \cite{FJRW}, Theorem 6.2.10, part
(3)). The first special case is to prove \eqref{FJRW=SG} when $g=0$,
$m=3$, and  $k_1=\cdots =k_m=0$. This however is straightforward,
because the explicit formulas for the FJRW invariants of this type are
already known (see property (iii) in Section \ref{sec:FJRW-inv}),
while for the SG-invariants we have 
\be\nn
\lbr \phi_i,\phi_j,\phi_k\rbr^{\rm SG} = \operatorname{Res}(\phi_i\phi_j\phi_k)
\ee
and the residue pairing is easy to compute. We leave the details as an
exercise. The second special case is the identity involving the
SG-correlator $\lbr x_2, x_2,x_2^{N-3},x_2^{N-2} \rbr_{0,4}^{\rm
  SG}$. In that case the RHS of \eqref{FJRW=SG} is
\be\nn
2^{\frac{2}{N-1}+\frac{N-3}{N-1}+\frac{N-2}{N-1}}
\lbr e_3, e_3, e_{2N-5}, e_{2N-3}\rbr^{\rm FJRW}_{0,4} = 
2^{1 +\frac{N-2}{N-1}} h^{-1} = -c^{-2} h^{-2},
\ee
where we used the formula from property (iv) in Section
\ref{sec:FJRW-inv}. Therefore, the identity follows from the following 
lemma:
\begin{lemma}\label{le:4pt-cor}
  The 4-point correlator
  \be\nn
  \lbr x_2, x_2, x_2^{N-3},x_2^{N-2}\rbr_{0,4}^{\rm SG} = -\frac{1}{h^2}.
  \ee
\end{lemma}
\begin{proof}
Let $f_t(x)=f(x)+t x_2$ be a deformation of $f$. The 4-point
correlator can be computed by
\be\label{4pt-cor}
\p_t\left.\Big(
  \lbr x_2,x_2^{N-3},x_2^{N-2}
  \rbr^{\rm SG}_{0,3}(t)\Big)\right|_{t=0},
\ee
where the correlator involved is a deformation of 
$\lbr x_2,x_2^{N-3},x_2^{N-2} \rbr^{\rm SG}_{0,3}$ constructed via flat (or
Frobenius) structure of Saito and Givental's higher genus
reconstruction.  On the other hand,
\be\nn
\lbr x_2,x_2^{N-3},x_2^{N-2}
\rbr^{\rm SG}_{0,3}(t) = \operatorname{Res}\,
\frac{    x_2\cdot x_2^{N-3}\cdot x_2^{N-2} }{
  \p_{x_1}(f_t) \p_{x_2}(f_t) \p_{x_3}(f_t) }\, dx  =
\frac{t}{N-1}\, \operatorname{Res}\,
\frac{    x_2^{N-2} \, dx }{
  \p_{x_1}(f_t) \p_{x_2}(f_t) \p_{x_3}(f_t) } ,
\ee
where $dx:=dx_1 \wedge dx_2\wedge dx_3$ and we used that
$x_2^{2(N-2)}= \tfrac{1}{N-1} x_2^{N-2} t$ in the local algebra
of $f_t$.  Since the Grothendieck residue of $x_2^{N-2}$ is
$-\tfrac{1}{2h}$, the above formula  and (\ref{4pt-cor}) yield the
formula that we have to proof. 
\end{proof}

In terms of the generating functions, we get
\be\nn
\mathcal{D}^{\rm FJRW}(\hbar,\mathbf{t}^{\rm FJRW} ) =
\mathcal{D}^{\rm SG} (\hbar c^2, \mathbf{t}^{\rm SG}),
\ee
where the formal variables are related by the following linear change:
\be\label{change:FJRW-SG}
t^{\rm FJRW}_{k,0} =-h\mathbf{i}\,  2^{\tfrac{N-2}{2N-2}} t^{\rm
  SG}_{k,N},\quad 
t^{\rm FJRW}_{k,2i-1} = 2^{\tfrac{i-1}{N-1}} t^{\rm SG}_{k,i},\quad
(1\leq i\leq N-1) .
\ee

\subsection{The 2-BKP hierarchy and the total descendant potential}\label{sec:BKP-D}

Let us explain the identification of $\mathcal{D}^{\rm
  SG}(\hbar,\mathbf{t}^{\rm SG})$ with a tau-function of the 2-BKP
hierarchy. This is done in three steps: construction of Hirota Bilinear Equations (HBEs) for the
total descendant potential 
$\mathcal{D}^{\rm SG}(\hbar,\mathbf{t}^{\rm SG})$, identifying the
HBEs with the HBEs of the principal Kac--Wakimoto hierarchy, and
finally identifying the Kac--Wakimoto hierarchy with the
$(h,2)$-reduction of the 2-BKP hierarchy. These
steps are considered respectively in \cite{GM},\cite{FGM}, and \cite{CM}.  

Let us recall the construction of HBEs from \cite{GM}.
Recall that the set of vanishing cycles in the Milnor lattice
$H_2(f^{-1}(1),\ZZ)$ of the $D_N$-singularity $f$ is a root system of
type $D_N$. Therefore, there exists an orthonormal basis $v_i$ ($1\leq
i\leq N$) of $H_2(f^{-1}(1),\QQ)$ with respect to the intersection
pairing, such that the set of vanishing cycles is given by 
$\pm(v_i \pm v_j)$. Furthermore, the Milnor lattice can be embedded in the dual
vector space of the local algebra $H_f$ via the following period map:
\be\nn
\Pi:H_2(f^{-1}(1),\ZZ)\to H_f^*,\quad
\lbr\Pi(\alpha),\phi_i\rbr:= \frac{1}{2\pi}
\int_\alpha \phi_i(x) \frac{dx}{df},
\ee
where $\tfrac{dx}{df}$ is a holomorphic 2-form $\eta$ defined in a tubular neighborhood
of $f^{-1}(1)$, such that $dx=df\wedge \eta$. Although the form
$\eta$ is not unique, its restriction to $f^{-1}(1)$ is unique and it
determines a holomorphic 2-form on the Milnor fiber $f^{-1}(1)$.

For each $\alpha\in H_2(f^{-1}(1),\ZZ)$ we have a multi-valued
analytic map $I_\alpha^{(-1)}: \CC\setminus{0}\to H_f$, such that 
\be\nn
(I^{(-1)}_\alpha(\lambda),\phi_i):=\frac{1}{2\pi} \int_{\alpha\subset
  f^{-1}(\lambda)} \phi_i(x) \frac{dx}{df},\quad \forall i=1,2,\dots, N,
\ee
where $(\ ,\ )$ is the Grothendieck residue pairing. Due to
homogeneity, we have the following simple formula
\be\nn
I^{(-1)}_\alpha(\lambda) = \sum_{i=1}^N \lambda^{m_i/h} \,
\lbr\Pi(\alpha),\phi_i\rbr\, \phi^i,
\ee
where $\{\phi^i\}$ is the basis of $H_f$ dual to $\{\phi_i\}$ with
respect to the residue pairing and
\be\nn
m_i:=\begin{cases}
  2i-1 & \mbox{ if } 1\leq i\leq N-1 \\
  N-1 & \mbox{ if } i=N.
  \end{cases}
\ee
coincide with the Coxeter exponents of the $D_N$ root system.

Let us choose an eigenbasis $H_i$ ($1\leq i \leq N$) for the Coxeter
transformation, satisfying $(H_i|H_j)=h\delta_{i,j^*}$, where $(\ |\
)$ is the invariant bilinear form of the $D_N$-root system and the
involution ${}^*$ is defined by
\be\nn
i^*=\begin{cases}
  N-i & \mbox{ for } 1\leq i\leq N-1, \\
  N & \mbox{ for } i=N.
  \end{cases}
\ee
More precisely, we choose $H_i$ to be the solutions to the following
system of equations:
\be\nn
v_i=\frac{\sqrt{2}}{h} \Big(
\eta^{m_1 i}H_1 +\cdots + \eta^{m_{N-1}i} H_{N-1}\Big),\quad 1\leq
i\leq N-1, \\
v_N=\frac{1}{\sqrt{h}} H_N,
\ee
where $\eta=e^{2\pi\ii/h}$. 
The Coxeter transformation $\sigma$ corresponds to analytic continuation
around $\lambda=0$, that is, the analytic continuations of
$I^{(-1)}_\alpha(\lambda)$ around $\lambda=0$ is
$I^{(-1)}_{\sigma(\alpha)}(\lambda)$. Therefore,
\be\nn
I^{(-1)}_{H_i}(\lambda) = \frac{\ii}{\sqrt{2}}
\frac{\lambda^{m_i/h}}{m_i/h} \rho_i \phi^i,\quad 1\leq i\leq N-1,\\
I^{(-1)}_{H_N}(\lambda) = \ii\, \sqrt{h} 
\frac{\lambda^{m_N/h}}{m_N/h} \rho_N \phi^N,
\ee
where $\rho_i$ ($1\leq i\leq N$) are some non-zero constants.

\begin{lemma}\label{le:ci}
There exists a choice of the orthonormal basis 
$v_i$ ($1\leq i\leq N$), such that, the constants 
$\rho_i=-\ii \xi^{m_i}$ ($1\leq i\leq N$),
where $\xi=e^{\pi\ii/h}$. 
\end{lemma}
\begin{proof}
The period map image of the Milnor lattice of the 
$D_N$-singularity is 
computed explicitly in \cite{MZ}, Proposition 5. In order to quote the
result, we have to switch to
different linear coordinates on $\CC^3$, that is, put 
$x_1= \xi^{-1}\, X_1$, 
$x_2=\xi^2 \, X_2$, 
$x_3=X_3$. 
Then $f= X_1^2 X_2+X_2^{N-1}+X_3^2$. According to \cite{MZ}, Proposition 5, there
exists an orthonormal basis $v_k$ ($1\leq k\leq N$) of
$H_2(f^{-1}(1);\QQ)$, such that, 
\be\nn
I^{(-1)}_{v_k}(\lambda) =
\frac{\lambda^{\theta+1/2}}{\Gamma(\theta+3/2)} 
\, \xi\, \Psi(v_k).
\ee  
Here the extra factor $\xi$ (compared to formula (1) in \cite{MZ}) comes from
the relation $dx_1dx_2dx_3=\xi dX_1dX_2dX_3$, the linear operator $\theta:H_f\to H_f$
defined by $\theta(\Phi_i) := \Big(\tfrac{m_{N-i}}{h}-\tfrac{1}{2}\Big)\Phi_i$ is
the so-called {\em Hodge grading  operator},
\be\nn
\Psi(v_k) = 2\sum_{i=1}^{N-1} \eta^{m_i k} \Gamma(m_i/h)
\Phi_{N-i},\quad (1\leq k\leq N-1),
\ee
and $\Psi(v_N)=-\ii \Gamma(m_N/h)\Phi_N$, where $\Phi_i:=X_2^{i-1}$
($1\leq i\leq N-1$) and $\Phi_N:=2X_1$. Let us point out that, compared
to \cite{MZ}, here we changed the sign of $v_N$ and  that the notation
$v_k$ in \cite{MZ} corresponds to $\Psi(v_k)$ here. We get the
following formulas:
\be\nn
I^{(-1)}_{v_k}(\lambda) =
2\sum_{i=1}^{N-1} \eta^{m_i k} \, 
\frac{\lambda^{m_i/h}}{m_i/h}\,
\xi\, \Phi_{N-i}  =
2\sum_{i=1}^{N-1} \eta^{m_i k} \, 
\frac{\lambda^{m_i/h}}{m_i/h}\,
\xi^{-m_{N-i}}\, \phi_{N-i}
\quad 
(1\leq k\leq N-1)
\ee
and 
$I^{(-1)}_{v_N}(\lambda) = 
-2\ii \lambda^{1/2}\, \xi\, \Phi_N =
-2\ii \lambda^{1/2} \phi_N$.  
On the other hand, we have 
\be\nn
I^{(-1)}_{v_k}(\lambda) = 
\frac{\sqrt{2}}{h} 
\sum_{i=1}^{N-1} 
\eta^{m_i k} \frac{\ii}{\sqrt{2}} 
\frac{\lambda^{m_i/h}}{m_i/h}  \rho_i \phi^i = 
2\sum_{i=1}^{N-1} 
\eta^{m_i k} 
\frac{\lambda^{m_i/h}}{m_i/h}  \rho_i (-\ii )\phi_{N-i}, 
\ee
for $1\leq k\leq N-1$ and $I^{(-1)}_{v_N}(\lambda) = 2\ii
\lambda^{1/2} \rho_N \phi^N = -2\ii \lambda^{1/2} \rho_N \phi_N$, where  we
used that $\phi^i=-2h \phi_{N-i} $ ($1\leq i\leq N-1$) and 
$\phi^N=-\phi_N$. Comparing the two formulas for
$I^{(-1)}$, we get $\rho_i=\ii \xi^{-m_{N-i}}=\ii\xi^{-h+m_i} = -\ii
\xi^{m_i}$ for $1\leq i\leq N-1$ and $\rho_N=1$.
\end{proof}

Let us define the higher order periods $I_\alpha^{(n)}(\lambda) :=
\p_\lambda^{n+1}I_\alpha^{(-1)}(\lambda)$. 
The vertex operators in singularity theory (see \cite{GM}) are defined by
\be\nn
\Gamma^\alpha(\lambda) =\exp\Big(
\sum_{k=0}^\infty I^{(-k-1)}_\alpha(\lambda) (-z)^{-k-1}\Big)\sphat
\exp\Big(
\sum_{k=0}^\infty I^{(k)}_\alpha(\lambda) (-z)^{k}\Big)\sphat,
\ee
where the quantization rules are
\be\nn
(\phi^i(-z)^{-k-1})\sphat := q^{\rm SG}_{k,i}/\sqrt{\hbar},\quad
(\phi_i (-z)^k)\sphat := (-1)^{k+1} \sqrt{\hbar} \partial_{q^{\rm SG}_{k,i}}.
\ee
The formal variables $q^{\rm SG}_{k,i}$ are related to $t^{\rm SG}_{k,i}$ by
the {\em dilaton shift}: $t^{\rm SG}_{k,i} = q^{\rm SG}_{k,i}+
\delta_{k,1}\delta_{i,1}$. Note that $i=1$ is the index of the basis
vector $\phi_1=1$, that is, the identity in the local algebra $H_f$.
Using the explicit formulas for the periods, we get
\be\nn
I^{(-k-1)}_{v_i}(\lambda) =
\frac{\mathbf{i}}{h}\, \sum_{s=1}^{N-1}
\frac{ (\eta^i \lambda^{1/h})^{m_s+k h}}{
  \tfrac{m_s}{h} \left(\tfrac{m_s}{h} +1\right)\cdots \left(
    \tfrac{m_s}{h} +k\right)}\, \rho_s \phi^s,
\ee
\be\nn
I^{(k)}_{v_i}(\lambda) = 2\mathbf{i}\, \sum_{s=1}^{N-1}
(-1)^{k+1}  \tfrac{m_s}{h} \left(\tfrac{m_s}{h} +1\right)\cdots \left(
    \tfrac{m_s}{h} +k-1\right)\, (\eta^i \lambda^{1/h})^{-m_s-k h}\, \rho_{s^*}\phi_s,
\ee
where $1\leq i\leq N-1$. Similarly,
\be\nn
I^{(-k-1)}_{v_N}(\lambda) =
\mathbf{i}\, 
\frac{ \lambda^{\tfrac{1}{2}+k} }{
  \tfrac{1}{2} \left(\tfrac{1}{2} +1\right)\cdots \left(
    \tfrac{1}{2} +k\right)}\, \rho_N \phi^N
\ee
and
\be\nn
I^{(k)}_{v_N}(\lambda) = \mathbf{i}\, \sum_{s=1}^{N-1}
(-1)^{k+1}  \tfrac{1}{2} \left(\tfrac{1}{2} +1\right)\cdots \left(
    \tfrac{1}{2} +k-1\right)\, \lambda^{-\tfrac{1}{2}-k}\, \rho_N\phi_N.
\ee
According to Givental--Milanov \cite{GM} the total descendant
potential $\tau:=\mathcal{D}^{\rm SG}(\hbar;\qq^{\rm SG}) $ is a solution to the
following HBEs:
\begin{align}\nn
  \operatorname{Res}_{\lambda=\infty} \frac{d\lambda}{\lambda} \left(
  \sum_{\alpha\in R} a_\alpha
  \Gamma^{\alpha}(\lambda)\otimes \Gamma^{-\alpha}(\lambda) \right) \
  \tau\otimes \tau = \frac{N(h+1)}{12h}\, \tau\otimes \tau  + & \\
  \nn
  +\frac{1}{h}\sum_{k=0}^\infty\sum_{i=1}^N (m_i+k h)
  \Big(q_{k,i}^{\rm SG} \otimes 1 - 1\otimes q_{k,i}^{\rm SG} \Big)
  \Big(\frac{\partial}{\partial q_{k,i}^{\rm SG} }\otimes 1 -
  1\otimes \frac{\partial}{\partial q_{k,i}^{\rm SG} }\Big)
  \tau\otimes \tau,
\end{align}
where $R$ is the set of all vanishing cycles and the coefficients
$a_\alpha$ are given by some explicit formulas, which would not be
needed. Moreover, according to Givental \cite{Gi}, the total
descendant potential satisfies the Virasoro constraints, in particular the  {\em string equation} 
$L_{-1} \mathcal{D}^{\rm SG}(\hbar;\qq^{\rm SG})=0$, where
\be\nn
L_{-1}:=-\frac{1}{4h\hbar}
\sum_{i=1}^{N-1} q^{\rm SG}_{0,i} q^{\rm SG}_{0,i^*} -
\frac{1}{2\hbar} q^{\rm SG}_{0,N} q^{\rm SG}_{0,N} +
\sum_{k,i} q^{\rm SG}_{k+1,i} \p_{q^{\rm SG}_{k,i}}.
\ee
Following Frenkel--Givental--Milanov \cite{FGM}, we can identify the above HBEs
with the HBEs of the principal Kac--Wakimoto hierarchy (of type
D). Let us recall the definitions and work out the precise
identification. Put $[1,N]$ for the set of integers
$\{1,2,\dots,N\}$ and let $E_+ = [1,2]\times \ZZ^+_{\rm odd}$. Let
$i:E_+\to [1,N]$ be the function defined by $i(2,l)=N$ for all $l\in
\ZZ^+_{\rm odd}$ and $i(1,l)$ is defined to be the unique integer
$i\in [1,N-1]$, such 
that $l\equiv 2i-1 \ ({\rm mod}\ h)$. Let
$y_e(e\in E_+)$ be 
a set of formal variables. Let $m:E_+\to \ZZ_{>0}$ be the function
defined by $m(1,l) = l$ and $m(2,l) = l (N-1)$. The numbers $m(e)$
($e\in E_+$) coincide with the so-called {\em exponents} of the affine
Kac--Moody Lie algebra (of type $D$). The HBEs of the principal
Kac--Wakimoto hierarchy of type D have the form
\begin{align}\nn
  \operatorname{Res}_{\zeta=\infty} \frac{d\zeta}{\zeta} \left(
  \sum_{\alpha\in R} a_\alpha
  \Gamma^{\alpha}(\zeta)\otimes \Gamma^{-\alpha}(\zeta) \right) \
  \tau\otimes \tau = \frac{N(h+1)}{12h}\, \tau\otimes \tau+ & \\
  \nn
  +\frac{1}{h} \sum_{e\in E_+} 
 m(e) \Big(y_e\otimes 1 - 1\otimes y_e \Big)
  \Big(\frac{\partial}{\partial y_e }\otimes 1 -
  1\otimes \frac{\partial}{\partial y_e}\Big)
  \tau\otimes \tau,
\end{align}
where the coefficients $a_\alpha$ are the same as the coefficients
$a_\alpha$ in the HBEs in singularity theory (this is one of the main results
in \cite{FGM}) and the vertex operators 
\be\nn
\Gamma^\alpha(\zeta) := \exp\Big( \sum_{e\in E_+}
(\alpha | H_{i(e)^*}) y_e \zeta^{m(e)}\Big)
\exp\Big( \sum_{e \in E_+} (\alpha | H_{i(e)})
\partial_{y_e} \frac{\zeta^{-m(e)}}{m(e)}\Big).
\ee
Recall that we have $(v_i|H_s) = \sqrt{2}\eta^{-i m_s}$,
$(v_i|H_N)=(v_N|H_i)=0$ for $1\leq i, s\leq N-1$ and
$(v_N|H_N)=\sqrt{h}.$ Therefore, the vertex operators of the
Kac--Wakimoto hierarchy can be computed explicitly. Comparing with the
vertex operators in singularity theory, we get that they coincide
under the substitution $\lambda=\zeta^h$ and an appropriate rescaling
of the formal variables. More precisely, if
$m=2l+1$ ($l\geq 0$) is an odd
integer, then let us construct unique integers $k$ and $s$, such that 
\be\nn
l+1= (N-1)k +s,\quad k\geq 0,\quad 1\leq s\leq N-1.
\ee                                                                 
Then the substitution that identifies the vertex operators takes the
following form:
\be\label{y_1m}
y_{1,2l+1} = \frac{\mathbf{i} \rho_s}{h\sqrt{2\,\hbar} }\, 
\frac{ q^{\rm SG}_{k,s} }{
\tfrac{m_s}{h} 
\left(\tfrac{m_s}{h} +1\right)\cdots 
\left(\tfrac{m_s}{h} +k\right)}
\ee
and
\be\label{y_2m}
y_{2,2l+1} = \frac{\mathbf{i} \rho_N}{\sqrt{h\, \hbar} }\, 
\frac{ q^{\rm SG}_{l,N} }{
\tfrac{1}{2} 
\left(\tfrac{1}{2} +1\right)\cdots 
\left(\tfrac{1}{2} +l\right)}.
\ee
As we already mentioned in the introduction, the Kac--Wakimoto
hierarchy can be identified with the $(h,2)$-reduction of the
2-component BKP hierarchy. The identification consists of expressing
the generators of the principal Heisenberg algebra, which in the
Kac--Wakimoto representation are given by
\be\nn
H_{i(e),m(e)}:=\frac{\partial}{\partial y_e},\quad
H_{i(e)^*,-m(e)}:= m(e) y_e,\quad e\in E_+,
\ee
in terms of the operators $J^a_m$ defined by (\ref{Jam}). We have (see \cite{CM}, Section
1.4 for more details)
\be\nn
H_{i,m} = \frac{1}{\sqrt{2}} \, J^1_m,\quad i\in [1,N-1],\quad m\equiv
2i-1\ ({\rm mod}\ h)
\ee
and
\be\nn
H_{N,m(N-1)} = \sqrt{\frac{N-1}{2}}\, J^2_m.
\ee
Therefore, the following formulas
\be\label{KW-BKP}
t_{1,m}:= \sqrt{2}\, y_{1,m},\quad
t_{2,m}:=\sqrt{h}\, y_{2,m},\quad m\in \ZZ^+_{\rm odd}
\ee
provide an identification between the Kac--Wakimoto hierarchy and the
$(h,2)$-reduction of the 2-component BKP. Recalling the substitution
\eqref{y_1m}--\eqref{y_2m},  we get that the total
descendant potential $\mathcal{D}^{\rm SG}(\hbar,\qq^{\rm SG})$ can be
identified with a solution of the $(h,2)$-reduction of the 2-component
BKP via the following substitutions:
\be\label{t_1m}
q_{1,2l+1} = \frac{\mathbf{i} \rho_s}{h \sqrt{\hbar} }\, 
\frac{ q^{\rm SG}_{k,s} }{
\tfrac{m_s}{h} 
\left(\tfrac{m_s}{h} +1\right)\cdots 
\left(\tfrac{m_s}{h} +k\right)}
\ee
and
\be\label{t_2m}
q_{2,2l+1} = \frac{\mathbf{i} \rho_N}{ \sqrt{\hbar} }\, 
\frac{ q^{\rm SG}_{l,N} }{
\tfrac{1}{2} 
\left(\tfrac{1}{2} +1\right)\cdots 
\left(\tfrac{1}{2} +l\right)}.
\ee
We replaced here the dynamical variables of 2-BKP with $q_{a,2l+1}$,
because we would like to work out also how the dilaton shift is
transformed under the substitution. In order to do this, 
let us identify the standard dynamical variables $t_{1,2l+1}$ and
$t_{2,2l+1}$  of 2-BKP with $t^{\rm SG}_{k,i}$ by the same
substitutions (\ref{t_1m}) and 
(\ref{t_2m}), that is replace $q$ by $t$ in both formulas. In the
new variables the dilaton shift, i.e., the relation between $q_{a,m}$
and $t_{a,m}$, takes the form
\be\nn
t_{a,m}=q_{a,m} + \delta_{a,1} \delta_{m,2N-1}\, \frac{\mathbf{i}
  \rho_1}{\sqrt{\hbar}}\, \frac{h}{h+1}.
\ee
The string operator in the variables $t_{a,m}$ takes the form
\be\nn
L_{-1}= -\frac{\mathbf{i} \rho_1}{\sqrt{\hbar}} \, \p_{t_{1,1}}+ 
\sum_{i=1}^{N-1} \frac{(2i-1)(2i^*-1)}{4h} t_{1,2i-1} t_{1,2i^*-1} +
\frac{1}{8} t_{2,1} t_{2,1} +\\
+\sum_{m\in \z^{\rm odd}_+ } \Big(
\tfrac{m+h}{h} t_{1,m+h}\p_{t_{1,m}}+
\tfrac{m+2}{2} t_{2,m+2}\p_{t_{2,m}}
\Big).
\ee
In order to work in the settings of \cite{CM}, we have to set
$\sqrt{\hbar}=\rho_1$, that is, the unique tau-function
$\tau^{\rm CM}(\mathbf{t}_1,\mathbf{t}_2)$  of the
2-component BKP hierarchy satisfying $L_{-1}\tau^{\rm CM}=0$ coincides
with $\mathcal{D}^{\rm SG}(\rho_1^2, \mathbf{t}^{\rm SG})$. 
Recalling the mirror symmetry result of 
Fan--Jarvis--Ruan, we get that 
\be\nn
\tau^{\rm CM}(\mathbf{t}_1,\mathbf{t}_2) = 
\mathcal{D}^{\rm FJRW}(\rho_1^2/c^2,\mathbf{t}^{\rm FJRW}).
\ee
Note that  $\rho_1^2/c^2= 2^{2-\tfrac{1}{N-1}}\, h \, \eta$. Combining the
linear changes \eqref{change:FJRW-SG} and  \eqref{t_1m}--\eqref{t_2m},
and Lemma \ref{le:ci} we get that the 2-BKP time variables are related
to the FJRW-variables via the following linear substitutions:
\begin{align}
\nn
t_{1, hk+2i-1} = & \ \phantom{-}\frac{\ii}{h}\, (2^{-1/h} \, \xi)^{m_i-1}\, 
\frac{t^{\rm FJRW}_{k,2i-1} }{\tfrac{m_i}{h}
\left(\tfrac{m_i}{h}+1\right) \cdots 
\left(\tfrac{m_i}{h}+k \right)}\quad 
(1\leq i\leq N-1) \\
\nn
t_{2, 2k+1} = & -\frac{1}{h}\, (2^{-1/h}\, \xi )^{m_N-1}\, 
\frac{t^{\rm FJRW}_{k,0}}{\tfrac{m_N}{h}
\left(\tfrac{m_N}{h}+1\right) \cdots 
\left(\tfrac{m_N}{h}+k \right)},
\end{align}
where $\xi=e^{\pi\ii/h}$.

\section{Free fermions and Pfaffians}\label{S_3}
In this section we use the neutral free fermion description of BKP hierarchy, introduced in \cite{DJKM1,DJKM2}, to derive a Pfaffian expression for the tau-function of 2-component BKP hierarchy in the Miwa parametrization. 

Let us recall the set up from Section 1.3 in \cite{CM}. Namely, let
$\phi_a(k)$, $a=1,2$, $k\in \ZZ$ be a set of neutral fermions,
satisfying the commutation relations
\be\nn
\phi_a(k)\phi_b(l)+\phi_b(l)\phi_a(k)=(-1)^k \delta_{a,b}\delta_{k,-l}.
\ee
The fermionic Fock space $\mathcal{F}$ is generated from a vacuum
vector $|0\rangle$ by the action of the above fermions with the only constraints
\be\nn
\phi_a(k)|0\rangle=0\quad (k<0),\quad
(\phi_1(0)+\mathbf{i}\phi_2(0))|0\rangle =0.
\ee
The vector space $\F$ is equipped with a positive definite Hermitian
form $H(\ ,\ )$, uniquely determined by the properties: $H(|0\rangle,|0\rangle)=1$ and $\phi_a(k)^\dagger=(-1)^k\phi_a(-k)$. Here $T^\dagger$ is a Hermitian conjugate of $T$, which satisfies
\be\nn
H(T^\dagger U,V)=H(U,T V),\quad \forall U,V\in \mathcal{F}.
\ee
If $T\in
\operatorname{End}(\mathcal{F})$ is a linear operator, then we define
\be\nn
\langle v_1| T |v_2\rangle:= H(v_1,T(v_2)). 
\ee

\subsection{Boson--Fermion correspondence}
Put
\be\label{Jam:fermionic}
J^a_m=\sum_{k\in \z} (-1)^k :\phi_a(-k-m)\phi_a(k) :,\quad 
\ee
where the fermionic normal ordering is defined by $: ab :=ab-\langle 0|
ab|0\rangle$. These operators satisfy the following commutation
relations:
\be\nn
[J^a_k,J^b_l]=2k\delta_{k,-l}\delta_{a,b}.
\ee
Put
\be\nn
\phi_a(z) := \sum_{k\in \z} \phi_a(k) z^k.
\ee
Then we have
\be\label{nf-vo}
\phi_a(z)=Q_a
e^{\sum_{m\in \z^+_{\rm odd}}
  J_{-m}^a \frac{z^m}{m}}
e^{-\sum_{m\in \z^+_{\rm odd}}
  J_{m}^a \frac{z^{-m}}{m}},
\ee
where $Q_a:\mathcal{F}\to \mathcal{F}$ is the linear operator defined by
\be\label{Q-action}
Q_a\phi_a(k)=\phi_a(k)Q_a,\quad Q_a\phi_{3-a}(k)=-\phi_{3-a}(k)Q_a,\quad
Q_a|0\rangle = \phi_a(0)|0\rangle.
\ee
The operator $Q_1Q_2$ has eigenvalues $\pm \mathbf{i}/2$. Let
$\mathcal{F}_0$ be the eigensubspace corresponding  to eigenvalue
$\mathbf{i}/2$. Let $\tt=(\tt_1,\tt_2)$ be a pair of two sequences of
formal variables of the form $\mathbf{t}_a=(t_{a,1},t_{a,3},\dots)$.   
The Boson--Fermion isomorphism $\mathcal{F}_0\cong
\mathbb{C}[\![\mathbf{t}]\!]$ can be defined as follows
\be\label{eq_BF}
v\in \mathcal{F}_0\mapsto \tau(v,\mathbf{t}):=
\langle 0|
\exp\Big(
\frac{1}{2}\sum_{a=1,2}
\sum_{m\in \z^+_{\rm  odd}} t_{a,m} J^a_m
\Big) |v\rangle. 
\ee
We have
\be
mt_{a,m} \tau(v,\mathbf{t})  = \tau(J^a_{-m} v,\mathbf{t}), \\
2\partial_{t_{a,m}}\tau(v,\mathbf{t}) = \tau(J^a_mv,\mathbf{t}),\\
\Gamma_a(\tt, z)\tau(v,\mathbf{t}) = \tau(Q_a^{-1} \phi_a(z)v, \mathbf{t}),
\label{bf-iso-vo}
\ee
where $J^a_m$ are the fermionic operators (\ref{Jam:fermionic}) and 
\be\label{vo:bosonic}
\Gamma_a(\tt, z) =
\exp\Big(\sum_{m\in \z^+_{\rm odd}} t_{a,m} z^m\Big)
\exp\Big(-\sum_{m\in\z^+_{\rm odd}}
2\frac{\partial}{\partial  t_{a,m}} \frac{z^{-m}}{m} \Big)
\ee
are vertex operators. Note that the 3rd formula in (\ref{bf-iso-vo})
is a consequence of the preceding two ones and (\ref{nf-vo}). 
The fermionic definition of the 2-component BKP (2-BKP) hierarchy and its
$(h_1,h_2)$-reduction is given in terms of the following set of
bilinear operators:
\be\nn
\Omega_m:= \sum_{a=1,2}\sum_{k\in\z} (-1)^k
\phi_a(k)\otimes \phi_a(-k-m h_a),\quad m\in \ZZ.
\ee
Namely, a function $\tau\in \F_0$ is said to be a {\em tau-function} of
the 2-BKP hierarchy if $\Omega_0(\tau\otimes \tau)=0$. 
A function $\tau\in \F_0$ is said to be a tau-function of the
$(h_1,h_2)$-reduction of 2-BKP if $\Omega_m(\tau\otimes \tau)=0$ for
all $m\in \ZZ_{\geq 0}$. We will be interested in the case when $h_1=h=2N-2$
and $h_2=2$, where $N\geq 2$, that is, the $(h,2)$-reduction. 
\begin{remark}
The 2-component BKP hierarchy can also be described in terms of 1-component neutral fermions. The expression for the tau-function in this case is analogous to the expression for 2D Toda lattice hierarchy in terms of 1-component charged fermions, namely for the 1-component neutral fermions one has
\be\nn
\tau(\mathbf{t})=\langle 0| \exp\Big(
\frac{1}{2}
\sum_{m\in \z^+_{\rm  odd}} t_{1,m} J_m
\Big) G \exp\Big(
\frac{1}{2}
\sum_{m\in \z^+_{\rm  odd}} t_{2,m} J_{-m}
\Big)
 |0\rangle,
\ee
where $G$ is the corresponding group element.
 This representation allows to represent this tau-function as a square of 2D Toda tau-function \cite{LO}. It would be interesting to find a 2D Toda tau-function, corresponding to the main object of this paper, that is, the $D_N$ singularity solution of 2-BKP.
\end{remark}

\subsection{Miwa parametrization}
Suppose that $v\in \mathcal{F}_0$ is arbitrary and let $\tau(v,\mathbf{t})$
be the function corresponding to $v$ via the Boson--Fermion correspondence (\ref{eq_BF}). 
Then the tau-function in the Miwa parametrization $\tau (Z_1,Z_2):=\left.\tau(v,\mathbf{t})\right|_{t_{a,m} = -\frac{2}{m} \, \Tr(Z_a^{-m}) }$ takes the form
\be\nn
\tau (Z_1,Z_2) = \langle 0| : \Gamma_1(z_{1,1})\cdots \Gamma_1(z_{1,N_1})
\Gamma_2(z_{2,1})\cdots \Gamma_2(z_{2,N_2}):|v\rangle,
\ee
where the normal ordering puts all $J^a_m$ with positive $m$ to the
right of all $J^a_m$ with negative $m$ and slightly abusing the notation we denote by 
\be\nn
\Gamma_a(z):=Q_a^{-1}\phi_a(z)= 
\exp\Big(
\sum_{m\in \ZZ_{\rm odd}^+ } J^a_{-m} \frac{z^m}{m}
\Big)
\exp\Big(-
\sum_{m\in \ZZ_{\rm odd}^+ } J^a_{m} \frac{z^{-m}}{m}
\Big)
\ee
the image of the vertex operator (\ref{vo:bosonic}) under the
Boson--Fermion isomorphism (cf. (\ref{bf-iso-vo})). Let
\be\nn
\tilde{K}(z,w):=\frac{z-w}{z+w},
\ee
and
\be\nn
K(z,w):=\iota_{|z|>|w|}\tilde{K}(z,w),
\ee
where $\iota_{|z|>|w|}$ is the operation of Laurent series expansion in the
region $|z|>|w|$. Using the OPE formula
\be\nn
\Gamma_a(z)\Gamma_b(w) = K(z,w)^{\delta_{a,b}}: \Gamma_a(z)\Gamma_b(w) :
\ee
for $|z_{1,1}|>|z_{1,2}|>\dots |z_{1,{N_1}}|$ and  $|z_{2,1}|>|z_{2,2}|>\dots |z_{2,{N_2}}|$ we get
\be\nn
\tau (Z_1,Z_2)  =\frac{
\langle 0| \Gamma_1(z_{1,1})\cdots \Gamma_1(z_{1,N_1})
\Gamma_2(z_{2,1})\cdots \Gamma_2(z_{2,N_2})|v\rangle}{
\prod_{a=1,2}\prod_{1\leq i<j\leq N_a} K(z_{a,i},z_{a,j})}.
\ee
Using that $\Gamma_a(z)=Q_a^{-1}\phi_a(z)$ and recalling the
definition (\ref{Q-action}) of the operators $Q_a$ we get     
\begin{equation}
\begin{split}\nn
\tau (Z_1,Z_2) & =
\frac{\langle 0|
Q_1^{-1} \phi_1(z_{1,1})\cdots Q_1^{-1} \phi_1(z_{1,N_1})
Q_2^{-1} \phi_2(z_{2,1})\cdots Q_2^{-1} \phi_2(z_{2,N_2})
|v\rangle}{
\prod_{a=1,2}\prod_{1\leq i<j\leq N_a} K(z_{a,i},z_{a,j})} =\\
&=(-1)^{N_1N_2}\, 
\frac{\langle 0|
\phi_1(z_{1,1})\cdots \phi_1(z_{1,N_1})
\phi_2(z_{2,1})\cdots \phi_2(z_{2,N_2})
|Q_1^{-N_1}Q_2^{-N_2} v\rangle}{
\prod_{a=1,2}\prod_{1\leq i<j\leq N_a} K(z_{a,i},z_{a,j})}.
\end{split}
\end{equation}
Recalling the definition (\ref{Q-action}) of the operators $Q_a$
($a=1,2$) we get that they satisfy the following relations:
\be\nn
Q_1^2=Q_2^2=1/2,\quad Q_1Q_2+Q_2Q_1=0. 
\ee
Therefore
\be\nn
Q_1^{-N_1}Q_2^{-N_2} = Q_1^{-N_1-N_2} Q_1^{N_2} (2Q_2)^{N_2} =
(-1)^{N_2(N_2-1)/2} 2^{(N_1+N_2)/2} (2Q_1Q_2)^{N_2}.
\ee
Note that $(-1)^{N_2(N_2-1)/2} = \mathbf{i}^{N_2^2-N_2}$. 
By definition $\mathcal{F}_0$ is an eigensubspace with eigenvalue
$\mathbf{i}$ for $2Q_1Q_2$. The above identity yields
\be\nn
Q_1^{-N_1}Q_2^{-N_2} v = \mathbf{i}^{N_2^2} 2^{(N_1+N_2)/2} v.
\ee
Therefore, in the Miwa variables the tau-function takes the form
\be\label{Miwa-tau}
\tau(Z_1,Z_2) := B(Z_1,Z_2) \,
\langle 0|
\phi_1(z_{1,1})\cdots \phi_1(z_{1,N_1})
\phi_2(z_{2,1})\cdots \phi_2(z_{2,N_2})
|v\rangle,
\ee
where
\be\nn
B(Z_1,Z_2) = \frac{\mathbf{i}^{-N_1^2}\,  2^{(N_1+N_2)/2}}{
\prod_{a=1,2}\prod_{1\leq i<j\leq N_a}
K(z_{a,i},z_{a,j})},
\ee
or
\be\label{eq_B}
B(Z_1,Z_2)^{-1}= \langle 0|
\phi_1(z_{1,1})\cdots \phi_1(z_{1,N_1})
\phi_2(z_{2,1})\cdots \phi_2(z_{2,N_2})
|0\rangle.
\ee
\subsection{Pfaffian Wick's theorem}
The first step in the proof of Theorem \ref{t2} is to express the
tau-function of the 2-BKP hierarchy in terms of Pfaffians. For this purpose we need 
the following Pfaffian version of Wick's theorem, which is a direct
neutral fermion analog of Wick's theorem of charged free fermions
(see, e.g., \cite{AZ}).  
The idea of the proof is also similar to the charged fermions (KP hierarchy) case.

Suppose that $v\in \mathcal{F}_0$ is a solution to the bilinear
equation $\Omega_0(v\otimes v)=0$. According to
van de Leur and Kac \cite{KL} there is a linear operator $G\in
\operatorname{GL}(\mathcal{F})$ with $v=G|0\rangle$,  such that
$\Omega_0(G\otimes G)=(G\otimes G)\Omega_0$. In particular, we get
\be\nn
\sum_{a=1,2}\sum_{k\in \z}
(-1)^k \langle U|\phi_a(k) G |V\rangle
\langle U'|\phi_a(-k) G |V'\rangle =
\sum_{a=1,2}\sum_{k\in \z}
(-1)^k \langle U|G\phi_a(k) |V\rangle
\langle U'|G\phi_a(-k) |V'\rangle 
\ee
for any $U,U',V,V' \in  \mathcal{F}_0$.
Following \cite{AZ} we call this identity the {\em basic bilinear
  condition}.  Below we assume that $\langle 0|v\rangle\neq 0$.
  Let   
\be\nn
v_i=\phi_{b_i}(z_i)=\sum_{k\in \z} \phi_{b_i}(k) z_i^k,\quad
1\leq i\leq 2n,\quad b_i\in \{1,2\},
\ee
be a set of $2n$ fermionic fields.

\begin{proposition}\label{prop:pfaffian}
  Suppose that $v=G|0\rangle$ is a solution to $\Omega_0(v\otimes v)=0$.
 Then

a) The following identity holds
\be\nn
\sum_{a=1,2} \sum_{k\in \z} (-1)^k
\langle 0|v_{2n} \phi_a(k) |v\rangle \,
\langle 0|v_1\cdots v_{2n-1} \phi_a(-k) |v\rangle = 0.
\ee

b) The following recursion holds
\be\nn
\langle 0|v_1\cdots v_{2n}|v\rangle = \sum_{i=1}^{2n-1} (-1)^{i-1}
\frac{\langle 0|v_i v_{2n}|v\rangle}{\langle 0|v\rangle}\,
\langle 0|v_1\cdots v_{i-1} v_{i+1} \cdots v_{2n-1}|v\rangle .
\ee

c) The following formula holds
$$
\frac{\langle 0|v_1\cdots v_{2n}|v\rangle }{\langle 0|v\rangle} =
\operatorname{Pf}\Big(
(2\theta(j-i)-1)\frac{\langle 0| v_i v_j|v\rangle}{\langle 0|v\rangle} 
\Big)_{1\leq i,j\leq 2n},
$$
where
$$
\theta(m)=
\begin{cases}
  1 & \mbox{ if } m>0,\\
  \frac{1}{2} & \mbox{ if } m=0, \\
  0 & \mbox{ if } m<0,
\end{cases}
$$
is the Heaviside function. 
\end{proposition}
\begin{proof}
a)  
Note that $v_i^\dagger:=\phi_{b_i}(-z_i^{-1})$ is the Hermitian
conjugate of $v_i$.
Let us recall the basic bilinear condition for
$$
U=v_{2n}^\dagger
|0\rangle,\quad
U'=v_{2n-1}^\dagger\cdots v_1^\dagger|0\rangle,\quad
V'=V=|0\rangle.
$$
We just need to check that the RHS of the basic bilinear identity is
$0$. If $k>0$ then $\phi_a(-k)V' =0$. If $k<0$ then $\phi_a(k)V=0.$
Therefore, only the terms with $k=0$ do not vanish, i.e.,
$$
\langle U|G\phi_1(0)|0\rangle\,
\langle U'|G\phi_1(0)|0\rangle +
\langle U|G\phi_2(0)|0\rangle\,
\langle U'|G\phi_2(0)|0\rangle .
$$
The above expression vanishes because $\phi_1(0)|0\rangle =
-\mathbf{i}\phi_2(0)|0\rangle$. 

b)
The idea is to use the identity proved in part a) and move the
fermions $\phi_a(k)$ and $\phi_a(-k)$ to the left side of the
corresponding correlator using the commutation relations
\be\nn
v_i \phi_a(k)+\phi_a(k) v_i = (-1)^k\delta_{a,b_i} z_i^{-k}. 
\ee
Let us first do this with the first correlator, i.e., replace
\be\nn
v_{2n} \phi_a(k)=-\phi_a(k) v_{2n} + (-1)^k\delta_{a,b_{2n}} z_{2n}^{-k}.
\ee
We get that the following two sums are equal:
\be\label{sum-1}
\sum_{a=1,2}
\sum_{k\in\z} (-1)^k \langle 0|\phi_a(k) v_{2n}|v\rangle \,
\langle 0|v_1\cdots v_{2n-1} \phi_a(-k)|v\rangle
\ee
and
\be\label{sum-2}
\sum_{a=1,2}\sum_{k\in \z} \delta_{a,b_{2n}}\, z_{2n}^{-k}\, \langle 0|v\rangle\,
\langle 0| v_1\cdots v_{2n-1} \phi_a(-k) |v\rangle=
\langle 0|v\rangle\,
\langle 0| v_1\cdots v_{2n-1} v_{2n} |v\rangle .
\ee
Let us split the sum (\ref{sum-1}) into $2n$ parts according to the
RHS of the identity
\be\nn
v_1\cdots v_{2n-1} \phi_a(-k) = -\phi_a(-k) v_1\cdots v_{2n-1} +
\sum_{i=1}^{2n-1} (-1)^{k+i-1}\,  \delta_{a,b_i} \, z_i^k\, v_1\cdots v_{i-1}
v_{i+1}\cdots v_{2n-1} .
\ee
The first part of the sum is
\be\nn
-\sum_{a=1,2}\sum_{k\in\z} \langle 0|\phi_a(k) v_{2n}
|v\rangle
\langle 0|\phi_a(-k) v_1\cdots v_{2n-1} |v\rangle =0,
\ee
where the terms with $k\neq 0$ vanish, because either $\phi_a(k)$ or
$\phi_a(-k)$ annihilates the vacuum, while the remaining terms have
only $k=0$ and they add up to $0$ thanks to the identity
$\phi_1(0)|0\rangle =-\mathbf{i} \phi_2(0)|0\rangle$.
The remaining parts have the form
\be\nn
\sum_{a=1,2}\sum_{k\in \z} (-1)^{i-1} \delta_{a,b_i} \,
z_i^k\,
\langle 0|\phi_a(k) v_{2n} |v\rangle
\langle 0| v_1\cdots v_{i-1} v_{i+1}\cdots v_{2n-1} |v\rangle,
\ee
so the sum (\ref{sum-1}) turns into
\be\nn
\sum_{i=1}^{2n-1}
(-1)^{i-1}
\langle 0|v_i v_{2n} |v\rangle
\langle 0| v_1\cdots v_{i-1} v_{i+1}\cdots v_{2n-1} |v\rangle.
\ee
Comparison with (\ref{sum-2}) completes the proof of part b).

c)
Let $A$ be the $2n\times 2n$ skew-symmetric matrix whose
upper-triangular entries are defined by
\be\nn
a_{ij} :=\frac{\langle 0| v_i v_j|v\rangle}{\langle 0|v\rangle},\quad
1\leq i<j\leq 2n .
\ee
We argue by induction on $n$. For $n=1$ the matrix has the form
\be\nn
A=
\begin{bmatrix}
  0 & a_{12} \\
  -a_{12} & 0
\end{bmatrix}
\ee
and its Pfaffian is $a_{12}$. For $n>1$, we use that
\be\nn
\operatorname{Pf}(A) = \sum_{i=1}^{2n-1} (-1)^{i-1} a_{i,2n}
\operatorname{Pf}(A_{i,2n}),
\ee
where $A_{i,j}$ denotes the matrix obtained from
$A$ by removing both $i$-th and $j$-th rows and columns. Our inductive
assumption implies that
\be
\nn
\operatorname{Pf}(A_{i,2n}) =
\frac{\langle 0| v_1 \cdots v_{i-1} v_{i+1}\cdots v_{2n-1}|v\rangle }{
  \langle 0 |v\rangle}.
\ee
It remains only to recall part b).
\end{proof}

\subsection{Pfaffian formula for the tau-function}

Let us apply part c) of Proposition \ref{prop:pfaffian} to
compute the tau-function (\ref{Miwa-tau}). Let 
\be\nn
\label{phi_to_gamma}
\phi_{a,b}(z,w)  := \frac{1}{2} \left.\frac{
\Gamma_a(z)\Gamma_b(w)\tau(\mathbf{t}) }{
\tau(\mathbf{t}) } \right|_{\mathbf{t}=0}
\ee
for $a\leq b$.  We get
\be\label{tau-Pf}
\tau(Z_1,Z_2)=\tau(0)\, B(Z_1,Z_2)\, \operatorname{Pf}(\Phi(Z_1,Z_2)),
\ee
where
\be\nn
\Phi(Z_1,Z_2)=
\begin{bmatrix}
  \Phi^{11} & \Phi^{12}\\
  \Phi^{21} & \Phi^{22}
\end{bmatrix}.
\ee
Here $\Phi^{aa}$ ($a=1,2$) are skew-symmetric matrices whose upper
triangular entries are defined by
\be\label{phiaa}
\Phi^{aa}_{i,j} =\phi_{a,a}(z_{a,i},z_{a,j}), \quad 1\leq i<j\leq N_a,
\ee
$\Phi^{21}=-(\Phi^{12} )^T$, and the entries of
$\Phi^{12}$ are defined by 
\be\label{phi12}
\Phi^{12}_{i,j} ={\bf i} \phi_{1,2}(z_{1,i},z_{2,j}),
\quad 1\leq i\leq N_1,
\quad 1\leq j\leq N_2. 
\ee

The factor $B(Z_1,Z_2)^{-1}$ in the formula (\ref{tau-Pf}) can also be
expressed as a Pfaffian. To derive such an expression it is enough to apply part c) of Proposition  \ref{prop:pfaffian} to (\ref{eq_B}). Let
$\Phi_0(Z_1,Z_2)$ be the matrix corresponding to the vacuum, that is,
to $\tau(\tt)=1$. Then, for $|z_{1,1}|>|z_{1,2}|>\dots > |z_{1,{N_1}}|$ and  $|z_{2,1}|>|z_{2,2}|>\dots > |z_{2,{N_2}}|$ we have
\be\label{tau_assum}
\tau(Z_1,Z_2)=\tau(0)\,  \frac{\operatorname{Pf}(\Phi(Z_1,Z_2))}{\operatorname{Pf}(\Phi_0(Z_1,Z_2))}.
\ee

If $\tt=(\tt_1,\tt_2)$, where $\tt_a=(t_{a,m})_{m\in\z^{\rm odd}_{>0}}$
($a=1,2$), then we denote by $\tau(\tt_1-[z_1^{-1}],\tt_2-[z_2^{-1}])$
the function obtained from $\tau(\tt):=\tau(\tt_1,\tt_2)$ via the
translation $t_{a,m}\mapsto t_{a,m}-2z_a^{-m}/m$. Then
\begin{equation}\label{eq_phi}
\begin{split}
\phi_{1,1}(z,w)  &  =K(z,w) \frac{\tau(-[z^{-1} ]-[w^{-1}],0)}{2\tau(0)}, \\
\phi_{1,2}(z,w) & = 
\frac{\tau(-[z^{-1} ],-[w^{-1}])}{2\tau(0)}, \\
\phi_{2,2}(z,w) & 
=K(z,w) \frac{\tau(0,-[z^{-1} ]-[w^{-1}])}{2\tau(0)}.
\end{split}
\end{equation}
Let
\begin{align}\label{tildephi}
\nn
\tilde{\phi}_{1,1}(z,w)  & :=
\frac{1}{2}\frac{z-w}{z+w}\, \frac{\tau(-[z^{-1} ]-[w^{-1}],0)}{\tau(0)}, \\
\tilde{\phi}_{1,2}(z,w) & := 
\frac{1}{2}\frac{\tau(-[z^{-1} ],-[w^{-1}])}{\tau(0)}, \\
\nn
\tilde{\phi}_{2,2}(z,w) & := 
\frac{1}{2}\frac{z-w}{z+w}\, \frac{\tau(0,-[z^{-1} ]-[w^{-1}])}{\tau(0)},
\end{align}
so that ${\phi}_{a,b}(z,w)  = \iota_{|z|>|w|} \tilde{\phi}_{a,b}(z,w) $. 
The following proposition describes the difference between $\phi_{a,a}$ and $\tilde{\phi}_{a,a}$:
\begin{proposition}\label{prop_exp}
\be\nn
\tilde{\phi}_{a,a}(z,w)- \frac{1}{2}\, \frac{z-w}{z+w} \in {\mathbb C}[\![z^{-1},w^{-1}]\!],
\ee
moreover, the difference vanishes when $|z|=|w|=\infty$.
\end{proposition}
\begin{proof}
Note that a tau-function of the 2-BKP hierarchy depends only on odd variables $t_{a,2k+1}$. So, for $a= b$ the ratio of the tau-functions on the RHS of (\ref{tildephi}) is in ${\mathbb C}[\![z^{-1}+w^{-1},z^{-3}+w^{-3},\dots]\!]$, and all terms, which contain at least one of the variables $t_{a,2k+1}$, are proportional to
\be
\nn
\frac{z-w}{z+w}\left(\frac{1}{z^{2k+1}}+\frac{1}{w^{2k+1}}\right)=\left(\frac{1}{w}-\frac{1}{z}\right)\left(\frac{1}{z^{2k}}-\frac{1}{z^{2k-1}w}+\dots+\frac{1}{w^{2k}}\right).
\ee
Therefore, the only term singular at $z=-w$ comes from the constant term in the
tau-function. Moreover, the RHS of this equation vanishes when $|z|=|w|=\infty$.
\end{proof}
It is obvious that ${\phi}_{1,2}(z,w)=\tilde{\phi}_{1,2}(z,w).$ Thus, we have
\begin{corollary}\label{cor_phiphi}
\be\nn
\tilde{\phi}_{a,b}(z,w)- \frac{1}{2}\delta_{a,b}\tilde{K}(z,w)={\phi}_{a,b}(z,w)-\frac{1}{2} \delta_{a,b} K(z,w).
\ee
\end{corollary}

Now we can relax the assumptions $|z_{1,1}|>|z_{1,2}|>\dots > |z_{1,{N_1}}|$ and  $|z_{2,1}|>|z_{2,2}|>\dots > |z_{2,{N_2}}|$. Therefore for arbitrary tau-function of the 2-BKP hierarchy in the Miwa parametrization we have the following Pfaffian formula:
\begin{proposition}\label{tau_as_Pf}
Let $\tilde{\Phi}$ be the matrix defined in the same way as $\Phi$,
except that in the definitions (\ref{phiaa})--(\ref{phi12}) of the
entries $\phi$'s are replaced by $\tilde{\phi}$'s. Then
  \be\nn
\tau(Z_1,Z_2)=\tau(0)\frac{2^{(N_1+N_2)/2} \operatorname{Pf}(\tilde{\Phi}(Z_1,Z_2))}
{{\bf i}^{N_1^2} \prod_{a=1,2}\prod_{1\leq i<j\leq N_a} \frac{z_{a,i}-z_{a,j}}{z_{a,i}+z_{a,j}}}.
\ee
\end{proposition}
\begin{proof}
It is easy to show that the numerator vanishes when $z_{a,i}=z_{a,j}$ for some $a\in \{1,2\}$ and $i\neq j$. Thus the RHS is in ${\mathbb C}[\![z_{1,1}^{-1},\dots,z_{1,N_1}^{-1},z_{2,1}^{-1},\dots,z_{2,N_2}^{-1}]\!]$. Moreover, for  $|z_{1,1}|>|z_{1,2}|>\dots > |z_{1,{N_1}}|$ and  $|z_{2,1}|>|z_{2,2}|>\dots > |z_{2,{N_2}}|$ it coincides with (\ref{tau_assum}), which completes the proof.
\end{proof}

\section{Grassmannian point for the simple singularity of type D}\label{sec:Gr}

In this section we recall the Grassmannian description of the 2-BKP hierarchy
\cite{Sh} and construct the integral description of the point of the BKP Grassmannian for the
tau-function which governs the simple singularity of type D.

\subsection{BKP Grassmannian}

We follow the notation from \cite{CM}, Section 1. 
Let $V=\CC(\!(z^{-1})\!) \oplus \CC(\!(z^{-1})\!)$ be the vector space
of formal Laurent series in $z^{-1}$ with coefficients in $\CC^2$. For
$f(z)=(f_1(z),f_2(z)) \in V$ and $g(z)=(g_1(z),g_2(z)) \in V$ put
\be\nn
(f(z),g(z)) := \sum_{i=1,2} \operatorname{Res}_{z=0} f_i(z)g_i(-z)\frac{dz}{z}.
\ee
Note that $(\ ,\ )$ is a non-degenerate symmetric bilinear pairing on
$V$. Let us define
\begin{align}\nn
  U_0 & = \CC(e_1+\ii e_2) +\CC[z]z e_1 + \CC[z]z e_2,\\
  \nn
  V_0 & = \CC(e_1-\ii e_2) + \CC[\![z^{-1}]\!]z^{-1} e_1 +
        \CC[\![z^{-1}]\!]z^{-1} e_2,
\end{align}
where $e_1=(1,0)$ and $e_2=(0,1)$ is the standard basis of
$\CC^2$. Both $U_0$ and $V_0$ are maximally isotropic subspaces and we
have a direct sum decomposition $V=V_0\oplus U_0$. Let $\pi:V\to U_0$
be the projection along $V_0$. The big cell
$\operatorname{Gr}_2^{(0)}$ of the 2-BKP Grassmannian is the set of 
all linear subspaces $U\subset V$ satisfying the following two conditions:
\begin{enumerate}
\item[(i)]
  $\pi|_U:U\to U_0$ is an isomorphism.
\item[(ii)]
  $U$ is a maximally isotropic subspace.
\end{enumerate}
Recall that a subspace $U\subseteq V$ is said to be {\em isotropic} if
$(u_1,u_2)=0$ for all $u_1,u_2\in U$.  If $U$ is a maximal element in
the set of all isotropic subspaces of $V$, then $U$ is called
maximally isotropic.

Suppose now that $\tau(\tt)\in \CC[\![\tt]\!]$ is a
formal power series, such that $\tau(0)\neq 0$. Then we define
\be\label{wave-1}
\Psi(\tt,z) := \Psi^{(1)}(\tt,z) e_1 +\ii \Psi^{(2)}(\tt,z) e_2 \quad \in
\quad V[\![\tt]\!],
\ee
where
\be\label{wave-2}
\Psi^{(a)}(\tt,z) =\frac{\Gamma_a(\tt,z)
  \tau(\tt)}{\tau(\tt)}. 
\ee
Let $U_\tau\subset V$ be the subspace spanned by the coefficients of
the Taylor's series expansion of $\Psi(\tt,z)$ at $\tt=0$. According
to Shiota (see \cite{Sh}, Section 3.1), the formal power series
$\tau(\tt)$ is a tau-function of the 2-BKP hierarchy if and only if
$U_\tau\in \operatorname{Gr}_2^{(0)}$. Moreover, the map $\tau\mapsto
U_\tau$ is a one-to-one correspondence between the tau-functions of the
2-BKP hierarchy and the points of $\operatorname{Gr}_2^{(0)}$. If
$\tau(\tt)$ is a tau-function of the 2-BKP hierarchy, then the
corresponding $\Psi(\tt,z)$ defined by (\ref{wave-1})--(\ref{wave-2})
is called the {\em wave function}.

The main goal of this section is to prove Theorem \ref{t1} and to
construct the corresponding point in the Grassmannian
$\operatorname{Gr}_2^{(0)}$ in terms of steepest descent asymptotic of
certain integrals. 
Our proof of Theorem \ref{t1} is based on the notion of the {\em
  Kac--Schwarz operators} for a given 
$U\in \operatorname{Gr}_2^{(0)}$, that is, differential operators $a$
such that   
\be\nn
a \, U \subset U.
\ee
Such operators were introduced first in \cite{KS}, where they proved
to be very convenient for the investigation of the solutions of the KP hierarchy associated to the  simple singularities of type A.

\subsection{Wave function and quantum spectral curve}
Let us reformulate the statement of Theorem \ref{t1} in terms of the
Grassmannian $\operatorname{Gr}_2^{(0)}$. Following \cite{CM}, Section
1, let us introduce two operators $a$ and $b$ ($a=\ell_{-1}$ in the
notation of \cite{CM}) 
\be\nn
a=(a_1,a_2),
\ee
where
\be\label{aoper}
a_1:=-{\bf i}z+z^{-h}\left(\frac{z}{h}\frac{\p}{\p z}-\frac{1}{2}\right),\\
a_2:=\frac{1}{2z^2}\left(z\frac{\p}{\p z}-1\right),
\ee
are the first order differential operators and
\be\nn
b=(z^h,z^2)
\ee
acts by multiplication. We have the following proposition.
\begin{proposition}\label{prop:ref_t1}
Let $U\in \operatorname{Gr}_2^{(0)}$ be a subspace corresponding to a
tau-function $\tau$ of the 2-BKP hierarchy. Then 

a) $\tau$ is a
tau-function of the $(h,2)$-reduction if and only if $b  \,  U \subset U$.

b) $\tau$ satisfies the string equation (\ref{str_eqn}) if and only if $a  \,  U\subset
U$. 
\end{proposition}
Part a) is  Corollary 1b in \cite{CM} and part b) is Lemma 9 in
\cite{CM}. Therefore, in order to prove Theorem \ref{t1} we have to
prove that there exists a unique subspace 
$U\in \operatorname{Gr}_2^{(0)}$, such that, $a  \,  U\subset U$ and
$b  \,  U\subset U$, that is, $a$ and $b$ are Kac--Schwarz operators for
$U$. In fact, we will construct an explicit basis of $U$ in terms of
stationary phase asymptotics of certain steepest descent integrals.

To begin with, note that $a$ and $b$ satisfy the canonical commutation relation
\be\nn
[a,b]=1.
\ee
Let $\Psi=\Psi^{(1)}e_1+{\bf i} \Psi^{(2)} e_2 \in U$ be such that 
\be\label{psi12}
\Psi^{(a)}(z)=1+O(z^{-1}).
\ee
To describe it, let us introduce degree grading in the space of differential
operators $\CC[z,z^{-1}][\partial_z]$, such that, 
\be\nn
\deg z^{-1} = \deg \frac{\p}{\p z}=-1.
\ee
Then
\be\nn
a^{h+1}=\left((-{\bf i}z)^{h+1}+\frac{h+1}{h}(-{\bf i})^h z\frac{\p}{\p z}+\dots,\dots \right),
\ee
where by $\dots$ we denote the terms of negative degree. The Kac--Schwarz operator
\be\nn
A:=ba+\frac{1}{2}-{\bf i}^h a^{h+1}=\left(-z\frac{\p}{\p z}+\dots,\frac{z}{2}\frac{\p}{\p z}+\dots\right)
\ee
does not contain terms of positive degree, therefore
\be\nn
A \Psi=O(z^{-1})e_1+O(z^{-1})e_2.
\ee
The LHS belongs to $U$, the RHS belongs to $V_0$, and since by
definition $\pi|_U$ is an isomorphism, we conclude that
\be\label{qsceq}
A \Psi=0.
\ee
We refer to (\ref{qsceq}) as the  {\emph {quantum spectral curve}} equation. 
\begin{lemma}\label{lemma1}
  Suppose that $a$ and $b$ are Kac--Schwarz operators for some
  subspace $U\in \Gr_2^{(0)}$ and let $\Psi(\tt,z)$ be the corresponding 
  wave function.  The quantum spectral curve equation (\ref{qsceq})
  has a unique, up to normalization, solution in $V$. Being normalized, this solution has the asymptotics (\ref{psi12}) and coincides with $\Psi(0,z)$.
\end{lemma}
\begin{proof}
  Let us prove the uniqueness. Suppose that $\Psi(z)\in V$ is a solution to
  the quantum spectral curve equation. Then the leading term of the Laurent series expansion is a non-vanishing constant. 
  Substituting the Laurent series expansion in $z^{-1}$ of $\Psi(z)$ in (\ref{qsceq}) and comparing
  the coefficients in front of the powers of $z$, we get a recursion
  which uniquely determines  the coefficients of $\Psi(z)$.

  By definition both $\Psi(z)$ and $\Psi(0,z)$ belong to $U$. Note
  that their projections via $\pi:V\to U_0$ coincide (with $e_1+\ii
  e_2$). Since $U\in \Gr_2^{(0)}$ the projection $\pi|_{U}$ is an
  isomorphism, so $\Psi(z)=\Psi(0,z)$. 
\end{proof}

Put $f(x)=x^{2h+2}-(h+1) x^2$. Let $u_a=f(\xi_a)$, where $\xi_1=1$ and
$\xi_2=0$ are two critical points of $f$, that is $u_1=-h$, $u_2=0$. 
Then the components of the wave
function can be identified with the steepest descent asymptotics of the
following integrals:
\be\label{oi}
\Psi^{(a)} \sim c_a \sqrt{\frac{\lambda_a}{\pi}} \int_{\gamma_a}
e^{\lambda_a (f(x)-u_a) } dx,\quad \lambda_a\to \infty,\,\,\,\,\,\,\,\,\, a \in \{1,2\}.
\ee
Here $\lambda_a(z):=\tfrac{\mathbf{i} z^{h_a+h_a/h}}{h+1}$, that is,
\be\nn
\lambda_1(z):=\tfrac{\mathbf{i} z^{h+1}}{h+1},\,\,\,\,\,\, 
\lambda_2(z):=\tfrac{\mathbf{i} z^{2+\frac{2}{h} }}{h+1}.
\ee
The contours $\gamma_a$ ($a=1,2$) are chosen as follows. Let us denote
by $\mathbb{D}_a\subset \CC$ the disk with center at the critical
point $\xi_a$ and a sufficiently small radius, so that the Morse lemma
applies, i.e., there exists a holomorphic coordinate $X_a(x)$ in
$\mathbb{D}_a$, such that, $f(x)=u_a-\tfrac{X_a(x)^2}{2}$ for all
$x\in \mathbb{D}_a$.  Let 
\be\nn
\mathbb{D}_a^-:=\{
x\in \partial \mathbb{D}_a\ |\ 
\operatorname{Re}(\lambda_a (f(x)-u_a)<0
\},
\ee
where $\partial \mathbb{D}_a$ denotes the boundary of $\mathbb{D}_a$. 
Using the Morse coordinate $X_a$ it is easy to see that
$\mathbb{D}^-_a$ consists of two disconnected arcs. Let us choose
the integration path $\gamma_a$ to be a path in $\mathbb{D}_a$
whose endpoints are on $\mathbb{D}_a^-$.  The asymptotic expansion of
  the integral in (\ref{oi}) depends only on the homology class of
  $\gamma_a$ in $H_1(\mathbb{D}_a,\mathbb{D}_a^-;\ZZ) \cong \ZZ$.
  This fact is easy to prove by modifying the standard argument of the
  steepest descent method (see \cite{Sha}, Chapter 5, Section 16). In fact, there is a more
  general theory of asymptotic expansions which applies to our
  case -- see  Chapter 16 in \cite{AGV} for more details. Let us choose
  $\gamma_a$ to be such that its homology class is a $\ZZ$-basis of
  $H_1(\mathbb{D}_a,\mathbb{D}_a^-;\ZZ)$. Later on (see Lemma
  \ref{le:Ip-de} below) we will have to
  work with asymptotic expansions of integrals of the form 
\be\label{oi-pole}
\int_{\gamma_a} e^{\lambda_a (f(x)-u_a)} \varphi(x) dx,
\ee
where $\varphi(x)\in \CC[x^2,x^{-2}]$. Note that the integrand of
(\ref{oi-pole}) is a meromorphic form on $\CC$ with a possible pole at
$x=0$ and that its residue at $x=0$ is $0$. If $a=1$, then the
asymptotic expansion is obtained via the standard theory. In the case
$a=2$, since $\varphi(x)$ might have a pole at $0\in \mathbb{D}_2$, we
have to be a little bit more careful. It turns out that the usual
  asymptotic estimates work, that is, if $\gamma_2$ does not contain
  $0$, then there is a well defined asymptotic
  expansion which depends only on the homology class of $\gamma_2$ in 
$H_1(\mathbb{D}_2,\mathbb{D}_2^-;\ZZ)$. Let us agree that $\gamma_2$
is a contour, such that, it does not go through $x=0$ and its
homology class in $H_1(\mathbb{D}_2,\mathbb{D}_2^-;\ZZ)$ is a
$\ZZ$-basis. We need only to specify the orientations of $\gamma_a$. 
Here and below  the fractional powers of $\lambda$
are defined via the principal branch of the logarithm, e.g.,
$\sqrt{\lambda}:=e^{\tfrac{1}{2}\log \lambda}$. We fix the normalization constants 
\be\label{c12}
c_1=\mathbf{i}\sqrt{2h(h+1)},  \,\,\,\,\,\, c_2=\sqrt{h+1}
\ee
and the orientation of the contours $\gamma_a$ to be such that the asymptotic expansions have the form (\ref{psi12}). 
\begin{remark}
It is possible to replace the local contours $\gamma_a$ in
(\ref{oi}) with global ones, such that, the asymptotic expansion does
not change. Let us give an example of global contours, asymptotically equivalent to the
local ones. Let $\lambda_a$ be positive for $a \in
\{1,2\}$.  
Then one can replace $\gamma_1$ with a path which  goes from  the sector $\left(\frac{\pi}{4(h+1)},\frac{3\pi}{4(h+1)}\right)$  to the sector $\left(-\frac{3\pi}{4(h+1)},-\frac{\pi}{4(h+1)}\right)$ trough the point $x=1$.
The second contour $\gamma_2$ can be replaced by a path which goes
from the sector
$\left(\pi-\frac{3\pi}{4(h+1)},\pi-\frac{\pi}{4(h+1)}\right)$ to the
sector 
  $\left(\frac{\pi}{4(h+1)},\frac{3\pi}{4(h+1)}\right)$ and contains a
  half circle
$C_\epsilon=\{ x=\epsilon e^{\ii \theta}\ :\
-\pi<\theta <0\}$ 
 with a sufficiently small radius $\epsilon$. We will not use
 the global contours in this paper. They might be important if one is interested in
constructing an analytic matrix model. Unfortunately, we could not
establish the analyticity of our matrix model due to complications in
the asymptotic expansions of certain double integrals -- see formula (\ref{Phi_ab})
below. 
\end{remark}

\begin{lemma}\label{lemma33}
We have
$
\Psi^{(a)}\in{\mathbb Q}[\![({\bf i}z^{h+1})^{-a}]\!]
$
for $a=1,2$. 
\end{lemma}
\begin{proof}
  Since $\Psi^{(1)}$ and $\Psi^{(2)}$ are steepest descent asymptotics
  of integrals, we can identify them with the formal
  perturbative expansion of Gaussian integrals, that is,  
\begin{align}\nn
\Psi^{(1)} & =\frac{1}{\sqrt{2\pi}}\int_{-\infty}^\infty dy\,
             e^{-\frac{y^2}{2}-\frac{2}{h\alpha^2}\sum_{j=3}^{2h+2}\frac{(2h+1)!}{j!(2h+2-j)!}\left(\frac{\alpha
             y}{2}\right)^{j}},\\
  \nn
\Psi^{(2)} & =\frac{1}{\sqrt{2\pi}}\int_{-\infty}^\infty dy \, e^{-\frac{y^2}{2}+\frac{{\bf i}^h}{(h+1)(2z^2)^{h+1}}y^{2h+2}},
\end{align}
where the RHS of the first formula is interpreted as a formal power series in $\alpha^2={\bf
  i}z^{-h-1}/h$. The statement of the lemma follows from the above
formulas. 
\end{proof}

Let us find the first coefficients of the expansion for $a=1$:
\be\nn
\Psi^{(1)}=1+(h+2)(2h+1)\left(\frac{1}{24}\alpha^2+\frac{1}{1152}(2h^2+53h+50)\alpha^4-\right.\\
\left.-\frac{1}{414720}(556h^4-1972h^3-41853h^2-76492h-36164)\alpha^6+O(\alpha^8)\right),
\ee
while for $a=2$ we easily find an expression for all coefficients:
\be\nn
\Psi^{(2)}=1+\sum_{k=1}^\infty \left(\frac{{\bf i}^h}{(h+1)(2z^2)^{h+1}}\right)^k\frac{(2hk+2k-1)!!}{k!}.
\ee
For brevity, let us introduce the differential operator
\be\label{do:g}
g:=-\frac{{\bf i}^h a_2^{h+1}}{h+1}.
\ee
We will also write $g_z$ if we would like to emphasize that $g$ acts
on functions in $z$.  It is easy to see that
\be\label{psi_as_aop}
\Psi^{(2)}=e^{g}\cdot 1,
\ee
that is, $\Psi^{(2)}(z)=e^{g_z}\cdot 1.$


\subsection{Higher vectors}\label{sec_hi}

In this section we explicitly describe a basis  for the point $U\in \Gr^{(0)}_2$. Put
\be
\nn
\Psi_k:=({\bf i}a)^{k} \Psi.
\ee
Note, that this definition makes sense for all $k \in {\mathbb Z}$. Indeed, operator $a_1^{-1}$ is well defined on ${\mathbb C}(\!(z^{-1})\!)$. Operator $a_2^{-1}$ is well defined on ${\mathbb C}(\!(z^{-2})\!)$ by $a_2^{-1} z^{2m}=\frac{2}{2m+1}z^{2m+2}$ for $m \in {\mathbb Z}$, and  from Lemma \ref{lemma33}, $\Psi^{(2)}(z) \in {\mathbb Q}[\![z^{-2}]\!]$.
\begin{lemma}\label{le:Ip-de}
Let $I_p(\lambda)$ be the asymptotic expansion of the integral
$\int_\gamma e^{\lambda f(x)} x^{2p} dx$, where 
$\gamma=\gamma_1$ or $\gamma_2$ and $p\in \ZZ$. Then 
\be\nn
\Big(\lambda \partial_\lambda +\tfrac{2p+1}{2h+2}\Big) I_p(\lambda) =
-\lambda h I_{p+1}(\lambda).
\ee
\end{lemma}
The proof is a direct computation with integration by parts. The above
Lemma yields the following formulas:
\be\label{vir-psi_2}
\Psi_k^{(a)} \sim c_a \sqrt{\frac{\lambda_a}{\pi}} \int_{\gamma_a}
e^{\lambda_a (f(x)-u_a) } (z^{h_a/h} x^2)^{k} dx,\quad \lambda_a\to \infty,
\ee
Note that the RHS of formula (\ref{vir-psi_2})  makes
sense also for $k<0$. 

\begin{lemma}\label{lemma_45}
$\Psi_{k}(z)=\Psi^{(1)}_{k}(z) e_1 + \ii  \Psi^{(2)}_{k}(z) e_2 \in U$ for all $k\in {\mathbb Z}$.
\end{lemma}
\begin{proof}
For $k\geq 0$ the statement follows from the definition of the Kac--Schwarz operators. 
To prove the statement for the negative $k$
let us introduce the Kac--Schwarz operator 
\be\nn
c:=b-({\bf i} a)^h.
\ee
Using the quantum spectral curve equation (\ref{qsceq}) and the commutation relation $[a,b]=1$ we conclude that $\Psi_{-1}=  ({\bf i}a)^{-1} \Psi =-2{\bf i} c\,\Psi$, thus $\Psi_{-1}\in U$. From the same commutation relation it immediately follows that
\be\nn
\Psi_{-k-1}(z) := -\frac{2\ii}{2k+1} \, c \Psi_{-k} \in U.
\ee
\end{proof}
\begin{corollary}
For $k \in {\mathbb Z}$ we have
\be\nn
A\Psi_k =-k \Psi_k. 
\ee
\end{corollary}
Then
\begin{lemma}\label{lemma36}
For $k\in {\mathbb Z}$ the components of $\Psi_{k}(z)=\Psi^{(1)}_{k}(z) e_1 + \ii 
  \Psi^{(2)}_{k}(z) e_2$ have the following leading order 
  terms:
\be\nn
\Psi_k^{(1)}=z^k \left(1+O(z^{-1})\right),\\
\Psi_k^{(2)}=\frac{(2k-1)!!}{(2{\bf i}z^2)^k}\left(1+O(z^{-1})\right),
\ee
where for negative $k$ we define $(2k-1)!!:=\frac{(-1)^k}{(2|k|-1)!!}$.
\end{lemma}

\begin{proof}
The statement follows from (\ref{aoper}) and  (\ref{psi12}).
\end{proof}

It is also easy to see that
\be\nn
\Psi_k^{(1)}\in z^{k}{\mathbb Q}[\![{\bf i}z^{-h-1}]\!],\\
\Psi_k^{(2)}\in \left({\bf i }z^2\right)^{-k}{\mathbb Q}[\![z^{-2h-2}]\!].
\ee
Vectors $\Psi_k$ do not completely generate $U$. Indeed, $U$ should contain an element of the form
\be\nn
\Phi_1=O(z^{-1})e_1+\ii z(1+O(z^{-1}))e_2.
\ee
This element cannot be expressed as a linear combination of $\Psi_k$.
However, it has a simple form and can be determined explicitly. Indeed,
since $a$ is a Kac--Schwarz operator, we have $a\, \Phi_1 \in
U$. Note that
\be\nn
a\, \Phi_1 = O(1)e_1+O(z^{-2})e_2.
\ee
We claim that $a\, \Phi_1=0$. Indeed, the projection of $a\, \Phi_1$
along $V_0$ is proportional to $e_1+\ii e_2$ and since it belongs to
$U$, we get that $a\, \Phi_1$ is proportional to $\Psi$. Comparing the
powers of $z$ in the $e_2$-components of $a\, \Phi_1$ and $\Psi$, we
get that the proportionality coefficient must be $0$, that is,
$a\,\Phi_1 =0$ as claimed.

The differential equation $a\,\Phi_1 =0$ can be solved explicitly:  
\be\nn
\Phi_1=\ii z e_2.
\ee
For $k>0$ let us define 
\be\label{Phik}
\Phi_{k}=b^{k-1} \Phi_{1}= \ii z^{2k-1} e_2 \in U.
\ee
From this equation recalling Lemmas  \ref{lemma1}, \ref{lemma_45} we get that
\be\label{basisexp}
U=\spn\{\{\Psi_{k}, k\in {\mathbb Z}\}, \{\Phi_k, k>0\} \}.
\ee
Thus, we proved Theorem \ref{t1}. 
\begin{remark}
Let us note that
\be\nn
z \sim
 \frac{\sqrt{\bf i}}{\pi}\, c_2\,  \lambda_2^{1/2} \int_{\gamma_2}
e^{\lambda_2 f(x) } (z^{1/h}  x)^{-1} dx,\quad \lambda_2\to \infty,
\ee
so that the basis vectors (\ref{Phik}) can be generated by the asymptotic expansion of the integrals \\
$c_2 \lambda_2^{1/2} \int_{\gamma_2}
e^{\lambda_2 f(x) } (z^{1/h}  x)^{-2k-1} dx$. So, all basis vectors (\ref{basisexp}) can be described by the asymptotic expansion of the integrals over the contours $\gamma_a$.
\end{remark}

From (\ref{psi_as_aop}) and definition of $\Psi^{(2)}_m$ it immediately follows that for $m \in {\mathbb Z}$
\be\label{psik_as_aop}
\Psi^{(2)}_m=e^{g}\cdot \frac{(2m-1)!!}{(2{\bf i}z^2)^m}\\
=\frac{1}{(2{\bf i}z^2)^m}\left((2m-1)!!+\sum_{k=1}^\infty \left(\frac{{\bf i}^h}{(h+1)(2z^2)^{h+1}}\right)^k\frac{(2(m+hk+k)-1)!!}{k!}\right).
\ee

\section{Matrix integral}\label{sec:mi}

From now on $U\in \Gr_2^{(0)}$ will denote the unique subspace
invariant under the action of the operators $a$ and $b$. Let
$\tau(\tt)=\tau(\tt_1,\tt_2)$ and $\Psi(\tt,z)$ be respectively the
corresponding $\tau$-function and the corresponding wave function. Our goal is to
express the tau-function $\tau(Z_1,Z_2)$ in the Miwa variables in
terms of the steepest descent asymptotic of a certain matrix
integral. 

\subsection{Integral kernel for the Pfaffian entries}

Our next goal is to express the entries of the Pfaffian matrix
$\tilde{\Phi}(Z_1,Z_2)$ in terms of asymptotics of integrals. The computation of
the entries of the Pfaffian matrix amounts to computing the 
3 types of 2-point functions (\ref{tildephi}). In this section we will
find integral representations for their expansions (\ref{eq_phi}).

\begin{lemma}\label{le:symm}
The tau-function has the following symmetry:
$\tau(\tt_1,\tt_2)=\tau(\tt_1,-\tt_2)$.
\end{lemma}
\begin{proof}
Indeed, it is easy to check
that if $\tau(\tt_1,\tt_2)$ is a tau-function, then
$\tau(\tt_1,-\tt_2)$ is also a tau-function. The Virasoro constraints
are invariant under the inversion $\tt_2\mapsto -\tt_2$ (note that due
to the dilaton shift, the other inversion $\tt_1\mapsto -\tt_1$ does
not preserve the Virasoro constraints). The uniqueness of the reduced 
tau-function satisfying Virasoro constraints implies that
$\tau(\tt_1,\tt_2)=\tau(\tt_1,-\tt_2)$.
\end{proof}
The symmetry of the tau-function from this lemma yields
immediately the following symmetry of the two point function
$\phi_{2,2}$:
\begin{corollary}\label{cor:symm}
  The following symmetry holds: 
$
      \phi_{2,2}(-w,-z)=\phi_{2,2}(w,z).
$
\end{corollary}
We also have:
\begin{lemma}\label{lem_K_reg}
  Let $g$ be the differential operator \eqref{do:g}. Then
\begin{equation}
\begin{split}\nn
e^{g_z+g_w}\,  \frac{zw}{z^2-w^2}&= \frac{zw}{z^2-w^2},\\
 e^{g_z+g_w}\,  \frac{z^2+w^2}{z^2-w^2}  -  \frac{z^2+w^2}{z^2-w^2} &\in {\mathbb C}[\![z^{-2},w^{-2}]\!].
\end{split}
\end{equation}
Moreover, the difference in the last line vanishes as $|z|=|w|=\infty$.
\end{lemma}
\begin{proof}
Since
\be\nn
\frac{z-w}{z+w}=\frac{z^2+w^2}{z^2-w^2} -2 \frac{zw}{z^2-w^2},
\ee
and operator $a_2$ is even, to prove the statement it is enough to show that 
\begin{equation}
\begin{split}\nn
\left(a_2^{h+1}+\tilde{a}_2^{h+1}\right)\frac{z-w}{z+w}\in {\mathbb C}[\![z^{-2},w^{-2}]\!],
\end{split}
\end{equation}
where $\tilde{a}_2$ is the operator $a_2$, acting on the functions of variable $w$. Since $h$ is even, the statement of the lemma follows from identities $a_2^{h+1}+\tilde{a}_2^{h+1}=(a_2^{2h}-a_2^{2h-1}\tilde{a}_2+\dots+\tilde{a}_2^{2h})(a_2+\tilde{a}_2)$ and
\be\nn
\left(a_2+\tilde{a}_2\right)\frac{z-w}{z+w}=\frac{1}{2w^2}-\frac{1}{2z^2}.
\ee
\end{proof}

Put
\be\nn
r_m:=(-1)^m \theta(m),
\ee
where $\theta(m)$ is the Heaviside function, so that
\be\nn
K(z,w)=2 \sum_{m=0}^\infty r_m \left(\frac{w}{z}\right)^m.
\ee

The identities in next proposition are very interesting. They will
play a key role in the construction of our matrix model. These formulas indicate the special role of the basis (\ref{basisexp}).
\begin{proposition}\label{kern_basis}
The following formulas hold:
\begin{align}
\nn
\phi_{1,b}(z,w) & = \sum_{k=0}^\infty r_k \Psi_{-k}^{(1)}(z)\Psi_{k}^{(b)}(w),
\\
\nn
\phi_{2,2}(w,z) & = \sum_{k=0}^\infty r_k \Psi_{-k}^{(2)}(z)\Psi_{k}^{(2)}(w)-\sum_{k=0}^\infty \left(\frac{z}{w}\right)^{2k+1}.
\end{align}
\end{proposition}
\begin{proof}
  Recalling the definition of the wave
  function (\ref{wave-1})--(\ref{wave-2}), we get the following formulas for $b\in\{1,2\}$:
  \begin{align}\nn
    \phi_{1,b}(w,z) & = \Psi^{(b)}(\tt^\circ,z)\Psi^{(1)}(w)/2 ,\\ \nn
    \phi_{b,2}(z,w) &=\Psi^{(b)}(\tilde{\bf t}^\circ,z)\Psi^{(2)}(w)/2,
  \end{align}
  where
$\mathbf{t}^\circ$ and $\tilde{\tt}^\circ$ are defined by 
$t^\circ_{1,m}= -2w^{-m}/m$, $t^\circ_{2,m}=0$ and
$\tilde{t}^\circ_{1,m}=0$,
$\tilde{t}^\circ_{2,m}= -2w^{-m}/m$. 
We introduce
\begin{align}\nn
  X_1(w,z) & =\phi_{1,1}(w,z)e_1+\ii \phi_{1,2}(w,z) e_2,\\
  \nn
X_2(w,z) & =\phi_{1,2}(z,w)e_1 + \ii\phi_{2,2}(w,z) e_2.
\end{align}
Viewing $w$ as a parameter, we get that
\be\nn
X_a(w,z) \in U,\quad a=1,2.
\ee
Using that $U$ is spanned by $\Psi_k$ and $\Phi_k$, recalling the
asymptotics given by (\ref{Phik}) and Lemma \ref{lemma36}, and having
in mind that $\phi_{1,2}(z,w)\in{\mathbb C}[\![z^{-1},w^{-1}]\!]$, we get
that $X_a(w,z)$ can be decomposed as follows:
\be\label{Xdeco}
X_1(w,z)=\sum_{k=0}^\infty \alpha_{k}(w) \Psi_k(z),\\
X_2(w,z)=\sum_{k=0}^\infty\beta_{k}(w) \Psi_{-k}(z) +\sum_{k=1}^\infty\gamma_k(w) \Phi_k(z),
\ee
for some Laurent series $\alpha_k,\beta_k,\gamma_k\in
\CC(\!(w^{-1})\!)$. In particular,
\be\label{phibeta}
\phi_{2,2}(w,z)=\sum_{k=0}^\infty\beta_{k}(w) \Psi_{-k}^{(2)}(z)+\sum_{k=1}^\infty\gamma_k(w) z^{2k-1}.
\ee
Comparing it with Proposition \ref{prop_exp} and having in mind that $\Psi_{-k}^{(2)}$ is even, $\Psi_{-k}^{(2)}(-z)=\Psi_{-k}^{(2)}(z)$, we immediately conclude that 
\be\nn
\gamma_k(w)=-w^{1-2k}.
\ee
Let us consider 
\be\nn
M(w,z)=\sum_{k=0}^\infty\beta_{k}(w) \Psi_{-k}^{(2)}(z) - \frac{1}{2}e^{g_z+g_w}\, \iota_{|w|>|z|} \frac{w^2+z^2}{w^2-z^2}.
\ee
From Proposition \ref{prop_exp} and Lemma \ref{lem_K_reg} it follows that
\be\nn
\phi_{2,2}(w,z) - \frac{1}{2}  e^{g_z+g_w}\,  K(w,z) =M(w,z),
\ee
and that $M(w,z)\in{\mathbb C}[\![z^{-2},w^{-2}]\!]$ with vanishing constant term. Let us show that $M(w,z)=0$. Let us act on $M$ by the operator $e^{-g_z}$. From (\ref{psik_as_aop}) we have 
\be\nn
e^{-g_z}M(w,z)=\sum_{k=0}^\infty\beta_{k}(w) (-1)^k \frac{(2{\bf i}z^2)^k}{(2k-1)!!} - \frac{1}{2}e^{g_w}\, \iota_{|w|>|z|}  \frac{w^2+z^2}{w^2-z^2}.
\ee
Hence $e^{-g_z}M(w,z)\in{\mathbb C}[\![z^{2},w^{-2}]\!]$ with trivial constant term. In the same time $e^{-g_z}M(w,z)\in{\mathbb C}[\![z^{-2},w^{-2}]\!]$, hence $e^{-g_z}M(w,z)=0$.
The kernel of the operator $e^{-g_z}$ on ${\mathbb C}[\![z^{-2}]\!]$ is trivial, therefore $M(w,z)=0$, and
\be\label{eq_phi_gr}
\phi_{2,2}(z,w) = \frac{1}{2}  e^{g_z+g_w}\,   K(z,w).
\ee
Comparing its expansion with 
(\ref{phibeta}) we conclude that $\beta_k(w)=r_k\Psi_{k}^{(2)}(w)$, and from (\ref{Xdeco}) it follows that $\alpha_k(w)=r_k\Psi_{-k}^{(1)}(w)$. This completes the proof. 
\end{proof}

\begin{corollary}\label{Cor22}
\be\nn
\tilde{\phi}_{2,2}(z,w)=\frac{1}{2} e^{g_z+g_w} \frac{z-w}{z+w}
\ee
\end{corollary}
\begin{proof}
The statement follows from Corollary \ref{cor_phiphi}, (\ref{eq_phi_gr}) and Lemma \ref{lem_K_reg}.
\end{proof}

Substituting the integral representations of the basis elements  (\ref{vir-psi_2})  we get for  $b\in\{1,2\}$ the following formulas:
\begin{equation}
\begin{split} \label{2int_formal}
\phi_{1,b}(z,w) & =\frac{c_1 c_b}{2\pi}  \sqrt{\lambda \mu}
\iint_{\gamma_1\times \gamma_b} 
                  e^{\lambda (f(x)-u_1) + \mu (f(y)-u_b)} K(z x^2,w^{2/h_b} y^2) dx dy, \\
\phi_{2,2}(w,z) & = \frac{c_2^2}{2\pi}  \sqrt{\lambda \mu}
\iint_{\gamma_2\times \gamma_2} 
                  e^{\lambda f(x) + \mu f(y)}
                  K(z^{2/h}x^2,w^{2/h}y^2) dx dy -
                  \sum_{k=0}^\infty \left(\frac{z}{w}\right)^{2k+1}, 
\end{split}
\end{equation}
where $\lambda=\lambda_a(z)$ and $\mu=\lambda_b(w)$ in agreement with the choice of the contour of integration.


\subsection{Double integrals}\label{sec:di}

In this section we derive double integral expressions for $\widetilde{\phi}_{a,b}$. Note that in \eqref{2int_formal} we first expand in the powers of $z^{-1}$ and
then apply the steepest descent method. It turns out that if we
apply directly the steepest descent method, without expanding in
$z^{-1}$, then we will obtain exactly the 2-point functions
$\widetilde{\phi}_{a,b}$, including the extra term in the last equation of (\ref{2int_formal}).

We will be interested in asymptotic expansions near the critical
points $\xi_1=1$ and $\xi_2=0$ of $f(x)=x^{2h+2}-(h+1)x^2.$ The Taylor
series expansion of $f(x)$ at $x=\xi_a$ has the form
\be\nn
f(x) = u_a - X^2 + O(X^3),\quad X:= c_a (x-\xi_a),
\ee
where the constants $c_a$ are the same as in \eqref{c12}. We
define a formal series expansion 
$\Phi_{a,b}(\lambda,\mu)$ of the following double integrals,
\be\label{Phi_ab}
\frac{c_a c_b}{2\pi}
\sqrt{\lambda\mu} \iint_{\gamma_a(\lambda)\times \gamma_b(\mu)}
e^{\lambda (f(x)-u_a) + \mu(f(y)-u_b) } \tilde{K}( \lambda^{\frac{1}{h+1}}  x^2,\mu^{\frac{1}{h+1}}  y^2)
dx dy,
\ee
where the contour $\gamma_a(\lambda):= \xi_a +
\tfrac{1}{c_a\sqrt{\lambda}} \RR$ is the contour used in the steepest
descend method. Let us restrict the values of the parameter
  $\lambda$ as follows:  for the contour $\gamma_1(\lambda)$ we
  require $\lambda\in \RR_{<0}$, while for
  $\gamma_2(\lambda)$ we require $\lambda\in \ii \RR_{>0}$. For such a choice for every pair of contours $\gamma_a(\lambda)$ and
  $\gamma_b(\mu)$, the sum $\lambda^{1/(h+1)}+\mu^{1/(h+1)}$ does not
  vanish. 

There are two cases. First, if $(a,b)\neq (2,2)$, then the integrand
is regular on $\gamma_a\times \gamma_b$.
Using Taylor series expansion in a neighborhood of the critical point
$(x,y)=(\xi_a,\xi_b)$, we get that the integrand in \eqref{Phi_ab} has
an expansion of the form
\be\nn\label{exp-Phi_ab}
\frac{1}{2\pi}\sqrt{\lambda\mu}
e^{-\lambda X^2 -\mu  Y^2}\,
\sum_{k,l=0}^\infty
a_{k,l}^{a,b} (\lambda,\mu) X^k Y^l dX dY.
\ee
Here $X=c_a(x-\xi_a)$, $Y=c_b(y-\xi_a)$, and $a_{k,l}^{a,b}$ is a
polynomial expression in $\lambda^{1/(h+1)}$,
$\mu^{1/(h+1)}$, and
$(\lambda^{1/(h+1)}+\mu^{1/(h+1)})^{-1}$ (for $a=b=1$) or $\lambda^{-1/(h+1)}$ (for $a=1$, $b=2$). Then
$\Phi_{a,b}(\lambda,\mu)$ is defined by integrating termwise
the above expansion over $\ii\RR\times \ii\RR$. Since
\be\nn
\frac{1}{2\pi}\sqrt{\lambda\mu}
\iint_{(\ii\rr)^2}
e^{-\lambda X^2 -\mu  Y^2}
X^k Y^l dX dY =
\begin{cases}
  0 & \mbox{ if $k$ or $l$ is odd},\\
  \frac{1}{2\pi}
  \Gamma(\tfrac{k+1}{2})
  \Gamma(\tfrac{l+1}{2})
  \lambda^{-k/2}\mu^{-l/2} & \mbox{otherwise}, 
\end{cases}
\ee
we get
\be\label{phi_1b}
\Phi_{a,b}(\lambda,\mu) =
\frac{1}{2\pi} \sum_{k,l=0}^\infty
a_{2k,2l}^{a,b} (\lambda,\mu)
\Gamma(k+\tfrac{1}{2}) \Gamma(l+\tfrac{1}{2})
\lambda^{-k} \mu^{-l}.
\ee
Let us point out that the above formal procedure applies in the single
variable case too. In particular, except for the case $a=2$ and $k<0$,
the local cycles $\gamma_a$ in (\ref{vir-psi_2}) can be replaced with the infinite cycles
$\gamma_a(\lambda_a)$. However,
whenever we do this we loose the analyticity of the integrals, so we
have to interpret them formally via the steepest descend
method expansion. Having this remark in mind let us consider the 
expression \eqref{phi_1b} with $a=1$. We claim that it coincides with 
$\widetilde{\phi}_{1,b}(z,w)$, where
$\lambda=\tfrac{\ii  z^{h+1}}{h+1}$ and
$\mu=\tfrac{\mathbf{i} w^{h_b+h_b/h}}{h+1}$. Indeed,
$\widetilde{\phi}_{1,b}(z,w)$  is a series of the same type as
\eqref{phi_1b}. According to formula
\eqref{2int_formal}, both series coincide after expanding
each coefficient in the powers of $z^{-1}$. Therefore, from Corollary
\ref{cor_phiphi} we get
$\Phi_{a,b}(\lambda, \mu)=\widetilde{\phi}_{a,b}(z,w)$.

Suppose now that $a=b=2$. The integration kernel has a singularity
at $(x,y)=(\xi_2,\xi_2) = (0,0)$, but as we will see now, the
singularity is integrable. The integrand in \eqref{Phi_ab} 
can be written as
\begin{align}\label{Phi_22}
 \frac{c_2^2}{2\pi} \, \sqrt{\lambda\mu}\,
   e^{-(h+1)(\lambda x^2 + \mu y^2)} \, 
    e^{\lambda x^{2h+2} + \mu y^{2h+2}}
    \frac{
    \lambda^{1/(h+1)} x^2- \mu^{1/(h+1)} y^2 }{
    \lambda^{1/(h+1)} x^2+ \mu^{1/(h+1)} y^2} dx dy.
\end{align}
Note that the polynomial $\lambda x^{2h+2} + \mu y^{2h+2}$ is
divisible by $ \lambda^{1/(h+1)} x^2+ \mu^{1/(h+1)} y^2$, that
is, the denominator of the integral kernel. Therefore, if we expand
the exponential $e^{\lambda x^{2h+2} + \mu y^{2h+2}}=
1+O(\lambda x^{2h+2} + \mu y^{2h+2})$, then only the constant
term is not divisible by the denominator, and \eqref{Phi_22}
can be expanded as follows:
\be\label{exp-Phi_22}
\frac{1}{2\pi} \, \sqrt{\lambda\mu}\,
e^{-(\lambda X^2 + \mu Y^2)}\left(
  \frac{
    \lambda^{1/(h+1)} X^2- \mu^{1/(h+1)} Y^2 }{
    \lambda^{1/(h+1)} X^2+ \mu^{1/(h+1)} Y^2} +
  \sum_{k,l=0}^\infty
  a_{k,l}^{2,2} (\lambda,\mu) X^k Y^l \right) dX dY.
\ee
Here $X:=c_2(x-\xi_2)=\sqrt{h+1} x$, $Y:=c_2(y-\xi_2) = \sqrt{h+1}
y$, and $a_{k,l}^{2,2}$ is a polynomial in $\lambda^{1/(h+1)}$ and
$\mu^{1/(h+1)}$.
\begin{lemma}\label{le:int_sing}
For $z,w>0$ 
\be\nn\label{eq_di}
\frac{1}{2}\frac{z-w}{z+w}=
\frac{zw}{2\pi}\iint_{\rr_+^2}  \frac{x-y}{x+y}e^{-w^2x-z^2y} \frac{dx dy}{\sqrt{xy}} .
\ee
\end{lemma}
\begin{proof}
Let us switch to polar coordinates $x=r\cos \theta$,
$y=r\sin\theta$. The integral takes the form
\begin{align}\nn
  &
\frac{zw}{2\pi} \,
\int_0^{\pi/2} \int_0^\infty
e^{-(z^2 \cos\theta +w^2 \sin\theta)  r}
\frac{\cos\theta -\sin\theta}{\cos\theta+\sin\theta}
  \frac{dr d\theta}{\cos \theta \, \sin\theta}  = &  \\
  \nn
  &
= \frac{zw}{2\pi} \,
  \int_0^{\pi/2}
  \frac{\cos\theta -\sin\theta}{\cos\theta+\sin\theta}
  \frac{1}{z^2 \cos\theta +w^2 \sin\theta}
  \frac{d\theta}{\sqrt{\cos\theta\, \sin\theta}}. &  
\end{align}
Finally, using the substitution $u=\sqrt{\tan \theta}$, we get
\be\nn
\frac{z w}{\pi}\,
\int_0^{\infty}
\frac{1-u^2}{(1+u^2)(z^2+ w^2 u^2)}\, du.
\ee
The above integral is straightforward to compute, and it is equal to
$\tfrac{1}{2}\tfrac{w-z}{w+z}$. 
\end{proof}
Now we are ready to prove the following lemma.
\begin{lemma}\label{l_57}
  The following identity holds:
\be\nn
\widetilde{\phi}_{a,b}(z,w)=
(-1)^{\delta_{a+b,4}}\Phi_{a,b}(\lambda,\mu),
\quad
1\leq a\leq b \leq 2. 
\ee
\end{lemma}
\begin{proof}
It remains to prove that $\widetilde{\phi}_{2,2}(z,w) = -\Phi_{2,2}(\lambda,\mu)$. 
Let us consider the integrand in the last line of (\ref{2int_formal}) in the variables $X$ and $Y$
\be\nn
\frac{1}{2\pi} \, \sqrt{\lambda\mu}\,
e^{-(\lambda X^2 + \mu Y^2)}\left(
2 \sum_{k=0}^\infty r_k \left(\frac{ \lambda^{1/(h+1)} X^2}{ \mu^{1/(h+1)} Y^2} \right)^k+
  \sum_{k,l=0}^\infty
  a_{k,l}^{2,2} (\lambda,\mu) X^k Y^l \right) dX dY.
\ee
Comparing it to (\ref{exp-Phi_22}) from Corollary \ref{cor_phiphi} we see, that the statement of the lemma is equivalent to the identity
\be\nn
\frac{1}{2\pi} \, \sqrt{\lambda\mu}\,
\iint_{\left(e^{-\ii \pi/4}\rr\right)^2}
e^{-(\lambda X^2 + \mu Y^2)}
  \frac{
    \lambda^{1/(h+1)} X^2- \mu^{1/(h+1)} Y^2 }{
    \lambda^{1/(h+1)} X^2+ \mu^{1/(h+1)} Y^2}\, dXdY =
  \frac{1}{2} \,\frac{w-z}{w+z}.
\ee
This identity follows from Lemma \ref{le:int_sing} after the substitution $x=\ii |\lambda|^{1/(h+1)} X^2$,
$y=\ii |\mu|^{1/(h+1)} Y^2$.
\end{proof}

\subsection{Change of variables}
To make a contact with the standard form of the generalized Kontsevich
matrix integrals \cite{AM,KMMMZ,K,W,IZ} let us make a change in
(\ref{vir-psi_2}) of the integration variables ${\bf
  i}z^{h_a/h}x^2\mapsto y$ and of the parameters $\lambda_a\mapsto
\lambda_a(z):= \tfrac{\ii z^{h_a+h_a/h}}{h+1}$, 
\be\nn
\lambda_a(f(x)-u_a) \mapsto -z^{h_a} y +\frac{{\bf i}^h y^{h+1}}{h+1} +\delta_{a,1}  \frac{{\bf i} h z^{h+1}}{h+1}.
\ee
Let 
\be\nn
\chi_a(y,z):=\sqrt{\frac{h_a z^{h_a}}{2}} 
e^{ -z^{h_a} y +
\frac{{\bf i}^h y^{h+1}}{h+1} +
\delta_{a,1}  \frac{{\bf i} h z^{h+1}}{h+1}}
\ee
and 
\be\nn
d \mu_a(y,z):=\chi_a(y,z) \frac{d y}{\sqrt{y}}.
\ee 

Let us discuss the transformation of the integration contours. To
begin with, since we are interested only in the asymptotic expansion,
let us replace the locally defined contour $\gamma_a$ in \eqref{vir-psi_2} with the
infinite contour $\gamma_a(\lambda_a)$. 
In the case when $a=1$, we have $\lambda_1\in \RR_{<0}$
$\Rightarrow$   
$\gamma_1(\lambda_1)=\RR$. Let us choose the solution $z$ of the
equation $\lambda_1=\lambda_1(z)$, such that,  
$\operatorname{Arg}(z)=\tfrac{\pi}{2(h+1)}$. Then, the image of
the integration path $\gamma_1(\lambda_1)$ under the change $y=\ii z
x^2$ is a ray $\ii z \RR_{\geq 0}$ independent of $z$.  
We are interested only in the asymptotic expansion of the integral in the vicinity of the critical point $y=\ii z$, hence we define  the integration path $\gamma_1^*$ for the new variable $y$ to be $\gamma^*_{1}:=\ii z \RR$.

For the second contour, we have  
$\gamma_2(\lambda_2) =e^{-\pi \ii/4} \RR$. In this case $\lambda_2\in \ii
\RR_{>0}$, so the equation $\lambda_2=\lambda_2(z)$ has a unique
solution $z\in \RR_{>0}$. The image of $\gamma_2(\lambda_2)$ under the
change $y=\ii z^{2/h} x^2$ is $\gamma^*_2:=\RR_{\geq 0}$. Note that,
unlike the case of the other contour, here the change of variables yields a 2:1
covering $\gamma_2(\lambda_2)\to \gamma^*_2$. The reason for our choice
in this case comes from the fact that the symmetry
$x\mapsto -x$ of $f(x)$ preserves the critical point $\xi_2=0$.  
Therefore, all integrals below involving integration along $\gamma_2$ gain
an extra factor of 2 when written as integrals along $\gamma_2^*$. 
First of all, note that the asymptotics (\ref{vir-psi_2}) can be written uniformly as
\be\nn\label{vir-psi_a}
\Psi^{(a)}_k(z) \sim  \frac{\ii^{2-a}}{\sqrt{\pi}}
\int_{\gamma_a^*}
 (-{\bf i}y)^k d \mu _a(y,z), \quad z\to \infty,
\ee
where $k\in \ZZ$ for $a=1$ and $k\geq 0$ for $a=2$. Furthermore, let
us make the associated  change of the variables in the double integrals
\eqref{Phi_ab} with 
$1\leq a\leq b\leq 2$. We get that \eqref{Phi_ab} is
transformed into 
\be\nn
\frac{\epsilon_{a,b}}{2\pi} \iint_{\gamma_a^*\times \gamma_b^*}  \frac{x-y}{x+y}
d \mu_a(x,z) d \mu_b(y,w),
\ee
where $\epsilon_{1,1}=-1$, $\epsilon_{1,2}=\ii$, and
$\epsilon_{2,2}=1$. Thus, using Lemma \ref{l_57} we proved
\begin{lemma}\label{doubleint}
Suppose that $1\leq a\leq b\leq 2$. Then  
\be\nn
\widetilde{\phi}_{a,b}(z,w)\sim-\frac{\ii^{a-b}}{2\pi} \iint_{\gamma_a^*\times \gamma_b^*}  \frac{x-y}{x+y}\,
d \mu_a(x,z) d \mu_b(y,w) ,\quad z,w\to \infty.
\ee
\end{lemma}
Let us make two remarks about Lemma \ref{doubleint}. First, let us
recall that the meaning of the asymptotic equality $\sim$ is
that the steepest descent method expansion 
of the integral on the RHS of $\sim$ coincides with the formal series
on the LHS of $\sim$. Second, note that $\epsilon_{a,b}=-\ii^{a-b}(-1)^{\delta_{a+b,4}}$, so the sign by which
$\epsilon_{a,b}$ and $-\ii^{a-b}$ differ matches precisely the sign by
which $\phi_{a,b}(z,w)$ and $\Phi_{a,b}(\lambda,\mu)$ differ.

\subsection{Matrix integral}

The goal of this section is to prove Theorem \ref{t2}. Let us consider
a version of de Bruijn's Pfaffian theorem \cite{B}. Put 
\be\nn
A_{i,j}=\iint_{\gamma_i\times \gamma_j} R(x,y)\phi_i(x)\phi_j(y) dx dy
\ee
for $1\leq i<j\leq 2n$ and $A_{i,j}=-A_{i,j}$ otherwise. Here $\gamma_i$ are some contours.
Then
\begin{lemma}\label{le:Bruijn}
If $R(x,y)=-R(y,x)$, then
\be\nn
\operatorname{Pf} (A)=
\int_{\gamma_1}\dots\int_{\gamma_{2n}}
\operatorname{Pf} (R) \prod_{i=1}^{2n} \phi_i (x_i )dx_1\dots dx_{2n},
\ee
where $\operatorname{Pf} (R)$ is the Pfaffian of the matrix with
entries $R(x_i,x_j)$ ($1\leq i,j\leq 2n$). 
\end{lemma}
The lemma follows immediately  from the definition of Pfaffian.

Let us first recall Proposition \ref{tau_as_Pf}. The entries of the matrix $\tilde{\Phi}(Z_1,Z_2)$ are
described by Lemma \ref{doubleint}, that is, the matrix is divided
into 4 blocks and the $(i,j)$-entry in the $(a,b)$-block is given by
the asymptotic expansion 
\be\nn
\tilde{\Phi}^{ab}_{i,j}\sim
-\frac{1}{2\pi}\iint_{\gamma_{a}^*\times \gamma_{b}^*}
\frac{x-y}{x+y}\,
d \mu_a(x,z_{a,i}) d \mu_b(y,z_{b,j}).
\ee
Note that the power of $\ii$ of the expression in Lemma
\ref{doubleint} is in agreement with the extra factor of $\ii$ in the
definition of the $(1,2)$-block in formula (\ref{phi12}). Recall
that $N_1+N_2$ is even and 
let us apply Bruijn's formula from Lemma \ref{le:Bruijn}.
We get 
\be
\nn
\operatorname{Pf}(\tilde{\Phi}(Z_1,Z_2))\sim
\frac{1}{(-2\pi)^{\frac{N_1+N_2}{2}}}
\int_{{(\gamma_1^*)^{N_1}}}\int_{{(\gamma_2^*)^{N_2}}}
\operatorname{Pf} (\tilde{K})\, \prod_{j=1}^{N_1}d \mu_1 (x_j,z_{1,j})  \prod_{j=N_1+1}^{N_1+N_2}d \mu _2(x_j,z_{2,j-N_1}).
\ee
Recall that  $\tilde{K}_{i,j}=\tilde{K}(x_i,x_j)=\frac{x_i-x_j}{x_i+x_j}$. Using
the Schur Pfaffian formula
\be
\nn
\operatorname{Pf} (\tilde{K})=\prod_{i<j}^{N_1+N_2}\frac{x_i-x_j}{x_i+x_j},
\ee
and the skew-symmetry with respect to the permutations of $x_i$ and
$x_j$ for $i \neq j$, we get 
\be
\label{Pf-asymp}
\operatorname{Pf}(\tilde{\Phi}(Z_1,Z_2))\sim\frac{1}{(-2\pi)^{\frac{N_1+N_2}{2}}N_1!N_2!}\int_{{(\gamma_1^*)^{N_1}}}\, \int_{{(\gamma_2^*)^{N_2}}}
Q\,  \prod_{j=1}^{N_1}\frac{dx_j}{\sqrt{x_j}} \prod_{j=1}^{N_2} \frac{dy_j}{\sqrt{y_j}},
\ee
where we have denoted $y_j:=x_{N_1+j}$ and
\be\label{Q_integr}
Q:=\Delta^*_{N_1}(x) \Delta^*_{N_2}(y) \det_{i,j=1}^{N_1} \chi_1(x_i,z_{1,j}) \det_{i,j=1}^{N_2}\chi_2(y_i,z_{2,j})  \prod_{i=1}^{N_1}\prod_{j=1}^{N_2} \frac{x_i-y_j}{x_i+y_j}.
\ee
Here 
\be
\nn
\Delta^*_N(x):=\prod_{i<j}^N\frac{x_i-x_j}{x_i+x_j}.
\ee
Note that 
\be\nn
\det_{i,j=1}^{N_1} \chi_1(x_i,z_{1,j})= 
 \det \sqrt\frac{h Z_1^h}{2} \, 
e^{\operatorname{Tr}\left(
h \lambda_1(Z_1) + \frac{\ii^h x^{h+1} }{h+1} \right)}\, 
\det_{i,j=1}^{N_1} e^{-x_i z_{1,j}^h}
\ee
and 
\be\nn
\det_{i,j=1}^{N_2} \chi_2(y_i,z_{2,j})= 
\det Z_2 \, 
e^{\operatorname{Tr}\left(
\frac{
\ii^h y^{h+1} }{ h+1} \right)}\, 
\det_{i,j=1}^{N_2} e^{-y_i z_{2,j}^2},
\ee
where $\lambda_1(Z_1)=\tfrac{\ii Z_1^{h+1}}{h+1}$, $x=\diag
(x_1,x_2,\dots,x_{N_1}),$ and $y=\diag(y_1,y_2,\dots,y_{N_2})$.
Let us recall the Harish-Chandra--Itzykson--Zuber integral formula for
unitary matrices
\be\nn
\int_{U(N)} \left[d U\right] e^{-\Tr U A U^\dagger B}=
C(N) \frac{\det_{i,j=1}^N e^{-a_i b_j}}{\Delta_N (a) \Delta_N (b)}
\ee
where $C(N)$ depends only on $N$, $ \left[d U\right]$ is the Haar measure on $U(N)$, and
$\Delta_N (a) =\prod_{i<j}(a_j-a_i)$ is the Vandermonde determinant.
Using the above formula to rewrite the determinants in
(\ref{Q_integr}) as integrals over unitary groups, we get
\be
\nn
Q= \det \sqrt\frac{h Z_1^h}{2}  \det Z_2  e^{h \Tr\lambda_1(Z_1)+\frac{{\bf i}^h}{h+1}\left(\sum_{j=1}^{N_1} x_i^{h+1}+\sum_{j=1}^{N_2} y_i^{h+1}\right)}\\
\times\frac{1}{C(N_1)C(N_2)} \int_{U(N_1)} \left[d U_1\right] e^{-\Tr U_1 x U_1^\dagger Z_1^h} \int_{U(N_2)} \left[d U_2\right] e^{-\Tr U_2 y U_2^\dagger Z_2^2}\\
\times \Delta_{N_1}(x)  \Delta_{N_1}(Z_1^h )  \Delta_{N_2}(y)  \Delta_{N_2}(Z_2^2 ) \Delta^*_{N_1}(x) \Delta^*_{N_2}(y)   \prod_{i=1}^{N_1}\prod_{j=1}^{N_2} \frac{x_i-y_j}{x_i+y_j}.
\ee
Substituting the above expression for $Q$ in \eqref{Pf-asymp} and
recalling the definition of ${\mathcal H}_N$ and ${\mathcal H}_N^+$
equipped respectively with measures $\widetilde{\left[d X\right]}$ and $\widetilde{\left[d Y\right]} $ (see Section
\ref{sec:mm}), we get that $\operatorname{Pf}(\tilde{\Phi}(Z_1,Z_2))$
coincides with the
asymptotic expansion near $X={\bf i}Z_1$ and $Y= 0$ of the following
two-matrix integral  
\be
\nn
\widetilde{C}(h,N_1,N_2)\ 
 \Delta_{N_1}(Z_1^h ) \  \Delta_{N_2}(Z_2^2 ) \  (\det Z_1)^{h/2} \ \det Z_2  \\
 \times  e^{h \Tr\lambda_1(Z_1)}\ 
 \int _{e^{\frac{(h+2)\pi}{2(h+1)}\ii}{\mathcal H}_{N_1}}   {\widetilde{\left[d X\right]} \, e^{\Tr W(X,Z_1^h) } } \int _{{\mathcal H}_{N_2}^+}{\widetilde{\left[d Y\right]}\, e^{\Tr W(Y,Z_2^2)}} S(X,Y),
\ee
where $W$ is given by (\ref{eq_Wpot}), $\widetilde{C}(h,N_1,N_2)$ is a
numerical constant depending only on $h$, $N_1$, and $N_2$,
and we used the formula
\be\nn
\det(X\otimes I_N+I_N\otimes X) = \prod_{i=1}^{N} 2x_i \
\prod_{1\leq i<j\leq N} (x_i+x_j)^2,
\ee
where $I_N$ is the $N\times N $ unit matrix, and
$X=\diag(x_1,\dots,x_N)$ is a diagonal matrix. We also have
\be\nn
S(X,Y)=\prod_{i=1}^{N_1}\prod_{j=1}^{N_2} \frac{x_i-y_j}{x_i+y_j} =
\det \left(\frac{
    X\otimes I_{N_2}-I_{N_1}\otimes Y}{
    X\otimes I_{N_2}+I_{N_1}\otimes Y}
  \right).
\ee
The formula stated in Theorem \ref{t2} follows with 
normalization factor 
\be\label{formula-N}
\mathcal{N}=
C(h,N_1,N_2)\, 
\frac{
  \Delta_{N_1}^*(Z_1)\Delta_{N_2}^*(Z_2) }{
 \Delta_{N_1}(Z_1^h ) \  \Delta_{N_2}(Z_2^2 ) \  (\det Z_1)^{h/2} \ \det Z_2 },
\ee
where the numerical constant 
$C(h,N_1,N_2)= \ii^{N_1^2} 2^{-(N_1+N_2)/2}/
\widetilde{C}(h,N_1,N_2)$.

\section{Examples and further comments}

\ 

1)
We conjecture, that an extended open version of the $D_N$ generating
function can be obtained by a simple deformation of the measure given by 
\be\nn
\widetilde{\left[d X\right]}\, e^{n \Tr \log X},
\ee
where $n$ is a formal parameter. It would be interesting to compare such a matrix integral with the results of Basalaev and Buryak \cite{BB}.

2)
Let $\mathcal{D}(\hbar,\mathbf{t})$ be the total descendent potential
  of the $D_N$ singularity. Theorem \ref{t2} gives a matrix integral
  for the total descendent potential in the Miwa variables for
  $\hbar=\rho_1^2 = - e^{2\pi\ii/h}$, where
  $h=2N-2$. On the other hand, using  the dilation equation and the
  $L_0$-constraint \cite{CM}, or the mirror symmetry from
  Section \ref{Mirror} and the dimension constraint for
  FJRW-invariants from Section \ref{sec:FJRW-inv}, one can show that
  to recover the 
  $\hbar$ dependence it is enough to rescale the $t$ variables 
\be
\nn
\left.{\mathcal D}(\hbar, {\bf t}^{\rm SG} )=
{\mathcal D}(\rho_1^2, {\bf t}^{\rm SG} )
\right|_{
  t_{k,s}\mapsto
  (\sqrt{\hbar}/\rho_1)^{\frac{h}{h+1}\left(
k+\operatorname{deg}(\phi_s)-1\right) } t_{k,s}}.
\ee
Since $\tau^{\rm CM}(\mathbf{t})=
{\mathcal D}(\rho_1^2, {\bf t}^{\rm SG} ),$ where the relation between
$\mathbf{t}^{\rm SG}=(t_{k,s}^{\rm SG})$ and $\mathbf{t}=(t_{a,m})$ is
given by \eqref{y_1m}--\eqref{KW-BKP}, the total
descendant potential can be expressed in terms of the tau-function as
follows: 
\be\nn
\left.
{\mathcal D}(\hbar, {\bf t}^{\rm SG} )=
\tau( {\bf t} )
\right|_{
  t_{a,m}\mapsto
  (\sqrt{\hbar}/\rho_1)^{ \frac{mh}{h_a(h+1)}-1 } t_{a,m} }.
\ee
Let us define the Miwa parametrization of the total descendant
potential by 
\be\nn
{\mathcal D}(\hbar, Z_1, Z_2):=
\left.{\mathcal D}(\hbar, {\bf t}^{\rm SG} )\right|_{
t_{a,m} = -\frac{2}{m}  \frac{\sqrt{\hbar}}{\rho_1}  \Tr Z_a^{-m}}.
\ee
Then we have 
${\mathcal D}(\hbar, Z_1, Z_2)= 
\tau(
 (\sqrt{\hbar}/\rho_1)^{-\frac{1}{h+1}} Z_1, 
 (\sqrt{\hbar}/\rho_1)^{-\frac{h}{2(h+1)}}Z_2)$. 
Thus the total descendent potential in the Miwa parametrization can be
identified with the asymptotic expansion of the following integral: 
\be\nn
\frac{ 
e^{\frac{\rho_1h}{\sqrt{\hbar} } \Tr\lambda_1(Z_1)} 
}{
(\sqrt{\hbar}/\rho_1)^{\frac{N_1^2}{2}+\frac{N_2^2}{2}} {\mathcal N}
} 
\int _{\rho_1^{-\frac{1}{h+1}}e^{\frac{(h+2)\pi}{2(h+1)}\ii}{\mathcal H}_{N_1}}   \widetilde{\left[d X\right]} \,
e^{\frac{\rho_1}{\sqrt{\hbar} }\Tr W(X, Z_1^h)}  
\int _{\rho_1^{-\frac{1}{h+1}} {\mathcal H}_{N_2}^+} \widetilde{\left[d Y\right]}\,  
e^{\frac{\rho_1}{\sqrt{\hbar} }\Tr W(Y,Z_2^2)}S(X,Y),
\ee
where the above integral is obtained from the matrix integral in
Theorem \ref{t2} via rescaling 
$Z_a\mapsto Z_a (\sqrt{\hbar}/\rho_1)^{-\frac{h}{h_a(h+1)}}$ and changing
the integration variables via
$X:= (\sqrt{\hbar}/\rho_1)^{-\frac{1}{h+1}} X'$ and 
$Y:=(\sqrt{\hbar}/\rho_1)^{-\frac{1}{h+1}} Y' .$

3) The FJRW invariants are known to be rational numbers. Using the
identification \eqref{FJRW=SG} between the SG-correlators and the
FJRW-correlators and the Euler characteristic constraint (ii) for the
FJRW invariants (see Section \ref{sec:FJRW-inv}), we get that the
coefficients of the total descendant potential
$\mathcal{D}(\hbar,\mathbf{t}^{\rm SG})$, that is, the coefficients in front of
monomials in $\hbar$ and $\mathbf{t}^{\rm SG} $ are rational
numbers. Specializing $\hbar=\rho_1^2$ and using the Euler
characteristic constraint again, it is easy to see that the
tau-function in the Miwa parametrization satisfies the following
condition: $\tau(\ii Z_1,Z_2)\in \QQ[\![Z_1^{-1}, Z_2^{-1}]\!]$. 

4) The normalization factor $\mathcal{N}$ in Theorem \ref{t2} can be
represented by the following matrix integral:
\be
\nn
{\cal N}:=
\frac{1}{\prod_{i,j=1}^{N_1}(\ii z_{1,i}+\ii z_{1,j})^\frac{1}{2}}
\int _{e^{\frac{(h+2)\pi}{2(h+1)}\ii}{\mathcal H}_{N_1}}  \left[d X\right] \,
e^{-\frac{\ii}{2}\Tr  \sum_{k=0}^{h-1}XZ_1^kXZ_1^{h-k-1}}
\int _{{\mathcal H}_{N_2}^+} \widetilde{\left[d Y\right]}\,  e^{-\Tr YZ_2^2 } .
\ee

5) If $N_a=0$ for $a=1$ or $a=2$, the tau-function $\tau(Z_1,Z_2)$
reduces to a tau-function of 1-component BKP hierarchy. These
tau-functions can be described by one-matrix models:  

$\bullet$ $N_1=0$
\be\nn
\tau(0,Z_2)\sim \frac{\int _{{\mathcal H}_{N_2}^+} \widetilde{\left[d Y\right]}\,  e^{\Tr W(Y,Z_2^2) } }{ \int _{{\mathcal H}_{N_2}^+} \widetilde{\left[d Y\right]}\,  e^{-\Tr YZ_2^2 } }.
\ee

$\bullet$ $N_2=0$
\be
\nn
\tau(Z_1,0)\sim \frac{ e^{h \Tr\lambda_1(Z_1)}\int _{e^{\frac{(h+2)\pi}{2(h+1)}\ii}{\mathcal H}_{N_1}}  \widetilde{\left[d X\right]} \, e^{\Tr W(X,Z_1^h)  } }{\frac{1}{\prod_{i,j=1}^{N_1}(\ii z_{1,i}+\ii z_{1,j})^\frac{1}{2}} \int _{e^{\frac{(h+2)\pi}{2(h+1)}\ii}{\mathcal H}_{N_1}}   \left[d X\right] \, e^{-\frac{\ii}{2}\Tr \sum_{k=0}^{h-1}XZ_1^kXZ_1^{h-k-1}} }.
\ee
To describe the corresponding point of the Sato Grassmannian for $N_2=0$ it is not enough to restrict the operators $a$ and $b$ on the first component of 2-BKP Grassmannian. Namely, it is easy to see that in this case $b_1$ is a not a Kac--Schwarz operator. However,
$a_1$, $a_1 b_1$, $a_1 b_1^2-3/2 b_1$ are the Kac--Schwarz operators. We claim that these three operators generate the Kac--Schwarz algebra in this case.

6) 
The measure $\widetilde{\left[d X\right]}$ is a natural measure for the so called $O(1)$ matrix model introduced by Kostov \cite{Kos} and related to the Bures ensemble (see, e.g., \cite{FK}).
The denominator can be simplified with an auxiliary Hermitian matrix integral
\be\nn\label{auxil}
\frac{1}{\sqrt{\det\left(X\otimes I_N+I_N\otimes X\right)}}=\int_{{\mathcal H}_N} [d A] e^{-\Tr X A^2}.
\ee
Similarly, we can rewrite the interaction term $S(X,Y)$ as a Gaussian
integral. The integrals in Theorem \ref{t2} can be
rewritten using the above integral as follows:
\be
\nn
\int \left[d X\right] \int \left[d A\right]
e^{-\Tr\left(XZ ^{h_i} -\frac{{\bf i}^h}{h+1}X^{h+1}+X A^2\right)}=
\int \left[d X\right] \int \left[d A\right]
e^{-\Tr\left(XZ ^{h_i} +2\frac{{\ii}^{h/2}}{\sqrt{h+1}} A X^{h/2+1}+X A^2\right)},
\ee
where we made a change of the $A$ matrix variable, $A \mapsto
A+\frac{{\ii}^{h/2}}{\sqrt{h+1}}X^{h/2}$. This allows us to simplify
the integral. In particular, for $h=2$ the integral over $X$ is Gaussian and can be
computed explicitly.  If we ignore the coefficients for the moment, then the potential of the last integral is of the form
\be\nn
x z^{h_i} + \frac{1}{N} a x^{N} +xa^2 =\int (z^{h_i}+W(x,a)) dx,
\ee
where $W(x,y)=x^{N-1}y+y^2$ is the potential, corresponding to the
singularity $D^T_N$, see Section \ref{sec:FJRW-inv}. 
We expect that the potentials of the matrix integrals for other
FJRW theories, in particular for  $E_N$ case, can be obtained from the
corresponding polynomials $W$ of $E_N^T$ in a similar way.

7)
It may be interesting to consider the matrix integral from Theorem \ref{t2} for the negative values of  $h$, therefore the negative values of $N$. For the matrix model in the $A_N$ case they correspond, in particular, to the Brezin--Gross--Witten
model which was recently shown by Norbury \cite{N} to describe
interesting intersection theory on moduli spaces of Riemann surfaces. We expect that a suitable version of the matrix integral in Theorem \ref{t2} for the negative values of $h$ is also related to some interesting enumerative geometry invariants. 

8)
We expect, that the complete set of W-constraints for the $D_N$ singularities, earlier derived in \cite{BM}, can be obtained from the 
invariance of the matrix integral with respect to the arbitrary
holomorphic change of integration variable. The computation should be
similar to the derivation of the Virasoro constraint for cubic
Kontsevich integral \cite{KMMMZ}, but more involved. Moreover, these constraints should also follow from the Kac--Schwarz description and boson-fermion correspondence. 
It would be interesting to solve these constraints in terms of the
cut-and-join operators, similar to the $A_N$-case described
in \cite{A,Z}.

9)
Kontsevich matrix integral was derived by Kontsevich \cite{K} using
diagram technique, which  can be related to the Strebel differentials
and a cell decomposition of the moduli space. 
Here one can invert the logic, and expand the matrix integral getting
the combinatorial model in terms of ribbon graphs. 
The diagram interpretation of the coefficients in the generating series 
should lead to the combinatorial interpretation of FJRW invariants for
$D_N$-singularity case.

10)
Obtained construction can be used to derive Kontsevich type integrals for other 
interesting tau-functions of BKP and  multicomponent BKP hierarchies.  

11) Our matrix integral should have a ``discrete" counterpart, which is expected to be also given by a certain type of 2-BKP tau-functions
with a natural neutral fermion description. 
In particular, correlators involving neutral fermions were used also by Harnad--van de
Leur--Orlov to construct 5 types of tau-functions
$Z_i(\mu,\mathbf{t},\overline{\mathbf{t}})$ of 2-BKP (see \cite{HvLO},
Section 3).  The third type, that is, $i=3$ in the notation of Section
3 in \cite{HvLO}, in the case when $a^c(\mathbf{z})=1$, seems to be
comparable to our matrix model. Note however, that the total descendent
potential and the tau-function $Z_3(\mu,
\mathbf{t},\overline{\mathbf{t}})$ have completely different nature,
so they seem to be unrelated. In the case of the total descendent
potential, the tau-function in Miwa parametrization is expressed in
terms of integrals  
\eqref{Pf-asymp} via their stationary phase asymptotic expansions. On the
other hand, the tau-function
$Z_3(\mu,\mathbf{t},\overline{\mathbf{t}})$ does not involve Miwa
parametrization, it is an infinite sum of multiple integrals
$I_3(N,\mathbf{t},\overline{\mathbf{t}}) $ of all possible dimensions
$N$. If we change to Miwa variables, then the integrals
$I_3(N,\mathbf{t},\overline{\mathbf{t}}) $ (see \cite{HvLO}, Section 3) become integrals of
hypergeometric type and they are quite different from \eqref{Pf-asymp}. 
Nevertheless, it will be interesting to find out 
if our tau-function has a relation to the models of  \cite{HvLO}. Such a relation could
allow us to give an alternative proof of Theorem \ref{t2} based on
Givental's higher-genus reconstruction formalism.  We expect, that the integrals of the two types can be related by a multi-scaling limit, 
similar to the relation between ``discrete" and ``continuous" matrix integral descriptions of the $A_N$ case.

12) The Schur Q-functions constitute a natural basis for the expansion of the BKP tau-functions. It would be interesting to find the expansion of the tau-function for the simple singularity of type D.

\section{Conflict of interest statement}

On behalf of all authors, the corresponding author states that there is no conflict of interest.

\Addresses

\end{document}